\documentclass[11pt]{article}
\usepackage{amssymb,amsmath,amsthm, amscd,geometry}
\usepackage{graphicx, color}
\usepackage[shortlabels]{enumitem}
\usepackage{hyperref}
\usepackage{pdfsync}
\usepackage{multirow}
\usepackage[affil-it]{authblk}
\usepackage[explicit]{titlesec}
\usepackage{array}
\usepackage{enumitem}
\usepackage{verbatim}

\newtheoremstyle{lessspaceplain}
{.3\baselineskip\@plus.2\baselineskip\@minus.2\baselineskip}
{.2\baselineskip\@plus.2\baselineskip\@minus.2\baselineskip}
{\itshape}
{}
{\bfseries}
{.}
{ }
{}

\theoremstyle{lessspaceplain}
\newtheorem{theorem}{Theorem}[section]
\newtheorem{theoremDT}[theorem]{Theorem/Definition}

\newtheorem{lemma}[theorem]{Lemma}

\newtheorem{cor}[theorem]{Corollary}

\newtheorem{question}[theorem]{Question}
\newtheorem{prop}[theorem]{Proposition}
\theoremstyle{definition}
\newtheorem{defn}[theorem]{Definition}
\newtheorem{definition}[theorem]{Definition}
\newtheorem{example}[theorem]{Example}
\newtheorem{remark}[theorem]{Remark}
\newtheorem{notation}[theorem]{Notation}

\newtheorem{obs}[theorem]{Observation}

\newtheorem{introthm}{Theorem}
\newtheorem{introcor}[introthm]{Corollary}

\def\NN {{\mathbb N}}
\def\N {{\mathbb N}}
\def\ZZ {{\mathbb Z}}
\def\BB {{\mathcal B}}

\def\RR {{\mathbb R}}

\def\GG{{\mathcal G}}

\newcommand{\calS}{\mathcal{S}}

\newcommand{\lam}{\lambda}

\newcommand{\into}{\hookrightarrow}

\newcommand{\al}{\alpha}

\newcommand{\eps}{\varepsilon}
\newcommand{\bxi}{\boldsymbol{\xi}}

\newcommand{\out}{\textup{Out}(F_n)}
\newcommand{\ol}{\overline}

\newcommand{\wt}{\widetilde}
\newcommand{\isom}{\text{Isom}}

\newcommand{\B}{\mathcal{B}}

\newcommand{\vol}[1]{\text{vol}(#1)}
\newcommand{\Teich}{Teichm\"{u}ller }

\newcommand{\Lip}{\text{Lip}}
\newcommand{\sig}{\sigma}

\newcommand{\os}{\mathcal{X}_n}

\renewcommand{\ol}{\overline}

\newcommand{\Cos}{\hat{\os}}
\newcommand{\stab}[1]{\text{Stab}(#1)}

\newcommand{\fmk}{f_{m,k}}
\newcommand{\f}[1]{f_{#1}}
\newcommand{\co}{\colon}
\newcommand{\lip}[1]{\text{Lip}(#1)}
\newcommand{\plf}{\mathcal{PLF}}
\newcommand{\from}{\colon}

\newcommand{\clBin}{\overline{B}_{\text{in}}}
\newcommand{\clBout}{\overline{B}_{\text{out}}}
\newcommand{\Bin}{B_\text{in}}

\newcommand{\scos}{\mathcal{\hat{X}}_n^S}
\newcommand{\Xm}{X_\infty(m)}

\newcommand{\fsc}{\text{FS}_n}
\newcommand{\inj}{\text{Inj}}
\newcommand{\fl}[1]{\displaystyle \overrightarrow{\lim_{#1 \to \infty}}}

\makeatletter
\renewenvironment{proof}[1][\proofname]{\par
\vspace{-\topsep}
\pushQED{\qed}%
\normalfont
\topsep0pt \partopsep0pt 
\trivlist
\item[\hskip\labelsep
\itshape
#1\@addpunct{.}]\ignorespaces
}{%
\popQED\endtrivlist\@endpefalse
\addvspace{6pt plus 6pt} 
}
\makeatother

\titlespacing\subsection{0pt}{12pt plus 4pt minus 2pt}{-4pt plus 2pt minus 2pt}
\titlespacing\section{0pt}{12pt plus 4pt minus 2pt}{-2pt plus 2pt minus 2pt}
\setlist[enumerate]{noitemsep, topsep=-5pt}
\setlist[itemize]{noitemsep, topsep=-7pt}

\begin{document}
\title{The Metric Completion of
Outer Space} 

\author{Yael Algom-Kfir%
\thanks{Electronic address: \texttt{yalgom@univ.haifa.ac.il}}}
\affil{Mathematics Department, University of Haifa\\ Mount Carmel, Haifa, 31905, Israel}
\maketitle
\begin{abstract}We develop the theory of a metric completion of an asymmetric metric space. 
We characterize the points on the boundary of Outer Space that are in the metric completion of Outer Space with the Lipschitz metric. We prove that the simplicial completion, the subset of the completion consisting of simplicial tree actions, is homeomorphic to the free splitting complex. 
We use this to give a new proof of a theorem by Francaviglia and Martino that the isometry group of Outer Space is homeomorphic to $\out$ for $n \geq 3$ and equal to $\text{PSL}(2,\ZZ)$ for $n=2$.
\end{abstract}

Outer Space, defined by Culler and Vogtmann \cite{CV}, has become in the past decades an important tool for 
studying the group of outer automorphisms of the free group $\out$. 
It is defined as the space of minimal, free and simplicial metric trees with an 
isometric action of $F_n$ (see section \ref{secOuterSpace}). 
Outer Space, denoted $\os$, admits a natural (non-symmetric) 
metric: the distance $d(X,Y)$ is the maximal amount of stretching 
any equivariant map from $X$ to $Y$ must apply to the edges of $X$. The group 
$\out$ acts on Outer Space by isometries. As mentioned, this metric is non-symmetric 
and in fact $\frac{d(X,Y)}{d(Y,X)}$ can be arbitrarily large (see \cite{AKB} for a 
general theorem about the asymmetry of Outer Space). Moreover it is not proper in 
the sense that outgoing balls 
\[\overline{B_{\text{out}}}(X,r) = \{Y \mid d(X,Y) \leq r \}\] 
are not compact (See Proposition \ref{compactBalls}). One way to fix this is to symmetrize the metric, i.e. define 
$d_s(X,Y) = d(X,Y) + d(Y,X)$. Closed balls in the metric $d_s$ are compact thus 
resolving both problems. However, in symmetrizing we lose much of the insight that 
Outer Space provides into the dynamics of the action of $\out$ on $F_n$. 
Moreover, Outer Space with the symmetric metric is 
not a geodesic space \cite{FM}. Therefore we prefer to keep the non-symmetric metric and determine the metric completion of 
$\os$ with the asymmetric metric. This raises the general question of how to 
complete an asymmetric metric. 
This turns out to be an intereseting problem in and of itself. The new terms in Theorem \ref{ThmCompl} are defined in Section 1,
\begin{introthm}\label{ThmCompl}
For any forwards continuous asymmetric space $(X,d)$ which satisfies property (\ref{star}), there is a unique forwards 
complete asymmetric space $(\hat X, \hat d)$ 
and an isometric embedding $\iota: X \to \hat X$ 
such that $\hat X$ is semi forwards continuous with respect to $\text{Im} (\iota)$ and so that $\iota(X)$ is forwards dense in $\hat X$. 
\end{introthm}
We also prove
\begin{introcor}\label{isomXintro}
For any forwards continuous asymmetric space $(X,d)$ which satisfies property (\ref{star}), 
each isometry $f \from X \to X$ induces a unique isometry $F \from \hat X \to \hat X$. 
\end{introcor}
After the completion of this article, it was brought to my attention that the completion of an asymetric metric space (also called a quasi-metric space) has been previously addressed, see for example \cite{BvBR}. However, their definition of a forward limit is slightly different (see \cite{BvBR}[page 8]) as well as their catagorical approach. Our approach is a little more hands-on providing some interesting examples as well to show the necessity of the conditions in our theorems. Both approaches give the same completion and we believe each has its merits.

We then address the issue of completing Outer Space. The boundary of Outer Space $\partial \os$ is the space of all homothety classes of \emph{very small} $F_n$ trees (see definition \ref{defnVerySmall}) that are not free and simplicial.
Our main result is: 
\begin{introthm}\label{myB}
Let $[T]$ be a homothety class in $\partial \os$, the point $[T]$ is contained in the completion of $\os$ if and only if for any (equivalently for some) $F_n$ tree $T$ in the class $[T]$, orbits in $T$ are not dense and arc stabilizers are trivial. 
\end{introthm}

We show that the Lipschitz distance can be extended to the completion 
(allowing the value $\infty$) and that isometries of $\os$ uniquely extend 
to the completion. We distinguish the set of simplicial trees in the metric 
completion of Outer Space and refer to it as the \emph{simplicial metric 
completion}. The simplicial metric completion is related to a well 
studied complex called the complex of free splittings. We show that the simplicial metric completion is invariant under any isometry of the completion of Outer Space (see Proposition \ref{isomExtend}). 

The complex of free splittings of $F_n$, denoted $FS_n$ (and 
introduced by Hatcher \cite{HatFS} as the sphere complex) is the 
complex of minimal, simplicial, actions of $F_n$ on simplicial trees 
with trivial edge stabilizers (see definition \ref{defnSplittingComplex}). 
The group $\out$ acts on $FS_n$ by simplicial automorphisms. This 
action is cocompact but has large stabilizers. 
Aramayona and Souto \cite{AS} proved that $\out$ 
is the full group of automorphisms of $FS_n$.
The free splitting complex is one of the analogues of 
the curve complex for mapping class groups. For example, 
Handel and Mosher \cite{HMhyp} showed that the free splitting complex 
with the Euclidean metric is Gromov hyperbolic (this proof was later 
simplified by \cite{HilionHorbezFS}). We prove that 

\begin{introthm}\label{myC}
The simplicial metric completion of Outer Space with the Lipschitz topology is homeomorphic to the free splitting complex with the Euclidean topology. 
\end{introthm}

We prove that an isometry of Outer Space extends to a simplicial automorphism of $\fsc$ (Proposition \ref{HomeoPresSimp}) and use Aramayona and Souto's theorem to prove the following theorem of Francaviglia and Martino.

\begin{introthm}\cite{FMIsom}\label{myD}
The group of isometries of Outer Space is equal to $\out$ for $n\geq 3$ and to $\text{PSL}(2,\ZZ)$ for $n=2$.
\end{introthm}

Francaviglia-Martino prove Theorem \ref{myD} for both the Lipschitz metric and for the symmetric metric. 
The techniques in this paper only apply to the asymmetric case. 
However, once the theory of the completion of Outer Space is established, the proof of Theorem \ref{myD} is natural and relatively light in computations.

As highlighted in \cite{FMIsom}, an application of Theorem \ref{myD} is that if $\Gamma$ is an irreducible lattice in a higher-rank connected semi-simple Lie group, then every action of $\Gamma$ on 
Outer Space has a 
global fixed point (this follows from a result of Bridson and Wade \cite{BW} 
that the image of $\Gamma$ in $\out$ of is always finite). In lay terms $\Gamma$ cannot act on $\os$ in an interesting way. 

We remark that the Thurston metric on \Teich space, which inspired for the 
Lipschitz metric on Outer Space, is in fact a proper metric i.e. closed balls are 
compact. Walsh \cite{WalshHoroBoundary} proved that the isometry group of \Teich space with the 
Thurston metric is the extended mapping class group. His technique was to embed 
\Teich space into the space of functions and consider its horofunction boundary. 
However, he uses the fact that the \Teich space with the Thurston metric is proper, which is false in the case of Outer Space. 

I originally wrote a version of this paper in 2012 as part of my postdoctoral work. 
The original paper was rejected by a journal after an unusually long refereeing process and as often happens, I put it away for a long time. 
I wish to thank Ilya Kapovich and Ursula Hamenst\"{a}dt for encouraging me to look at it again and to extend the results of the original work. I also thank Emily Stark for thoughly reading the first section and  Mladen Bestvina for inspiring the original work. Thanks to David Dumas for pointing me to papers on quasi-metric spaces.

\section{The completion of an asymmetric metric space}

This section develops the theory of the completion of an incomplete asymmetric metric space (see Definition \ref{asymMetricSp}). Working out the details often reminded me of W. Thurston's quote in his famous ``notes'' \cite{ThurstonNotes} ``Much of the material or technique is new, and more of it was new to me. As a consequence, I did not always know where I was going \dots The countryside is scenic, however, and it is fun to tramp around \dots''. 

The starting point was to imitate the completion of a metric space using Cauchy sequences. This strategy eventually works but many potholes must first be overcome. 
For example, not every convergent sequence is a Cauchy sequence (Example \ref{ConvNoCauchy}). In fact, the Cauchy condition is a little unnaturally strict in the asymmetric setting so we relax it to define admissible sequences (see Definition \ref{add_seq_defn}). We define a limiting distance between two admissible sequences (Lemma \ref{dist_limit}). In the symmetric setting the completion is the space of equivalence classes of sequences with limiting distances equal to 0. In the asymmetric setting an admissible sequence $\{x_k\}$ can converge to a point $x$ but one of the limiting distances between $\{x_k\}$ and the constant sequence $\{x\}$ is not 0. Therefore we define a weaker equivalence relation (see Definition \ref{defnEquiv}). We then define the completion $\hat X$ of $X$ as the set of equivalence classes of admissible sequences with an appropriate notion of $\hat d$ (see Definition \ref{defHat}). However, one must add extra conditions to guarantee that $(\hat X, \hat d)$ is an asymmetric metric space (Lemma \ref{triDHatLemma}), and that it is complete (Lemma \ref{bigLemma}). Finally, under certain conditions we show that for $(X,d)$ an asymmetric metric space $(\hat X, \hat d)$ is the unique complete asymmetric metric space that contains a dense copy of $X$ (Theorem \ref{ThmCompl}), and that isometries of $X$ uniquely extend to $\hat X$ (see Corollary \ref{isomX}). 
The following are not strictly needed for the rest of the paper but are included for clarity: Proposition \ref{hatXsatisfies11}, Remark \ref{remarkBadExample} and Lemma \ref{connection}.

\begin{definition}\label{asymMetricSp}
An asymmetric metric on a set $X$ is a function $d:X \times X \to \RR \cup \{\infty\}$ which satisfies the following properties:
\begin{enumerate}
\item $d(x,y) \geq 0$ 
\item If $d(x,y) = 0$ and $d(y,x) = 0$ then $x=y$.
\item For any $x,y,z \in X$, $d(x,z) \leq d(x,y) + d(y,z)$.
\end{enumerate}
\end{definition}

Since the metric is not symmetric, a sequence might have many (forward) 
limits.

\begin{defn}\label{dfCFL} The point $x \in X$ is a \emph{forward limit} of the sequence 
$\{x_n\}_{n=1}^\infty \subset X$ if 
\[\lim_{n \to \infty} d(x_n,x)=0,\] and it is the \emph{closest forward limit} (CFL) if additionally, for every $y$ in $X$ 
such that $\lim_{n \to \infty} d(x_n,y)=0$ we have $d(x,y)=0$. If the two conditions are satisfied then 
we write $ x = \fl{n} x_n$. 
\end{defn}
Note that if $\{x_n\}$ admits a CFL $x$ then it is unique since if $z$ 
was another such limit we would get $d(x,z) = 0 = d(z,x)$ hence $x=z$. 

A closest 
backward limit (CBL) is defined as in Definition \ref{dfCFL}, only switch the order of the parameters of $d$.

We would like to define the completion of an asymmetric metric 
space in a similar fashion to 
the completion of a metric space, as equivalence classes of (appropriately defined) 
Cauchy sequences. 
We introduce the notion of an admissible sequence which is better suited to the asymmetric setting. 

\begin{defn}\label{add_seq_defn}[Admissible sequences, Cauchy sequences]
A sequence $\xi = \{x_n\}$ is \emph{forwards Cauchy} if for all $\eps>0$ there is an $N(\eps) \in \NN$ such that 
\begin{equation}\label{eq21} d(x_i, x_j)< \eps \quad \text{for all} \quad j>i>N(\eps)\end{equation}
In words, we first choose the left index to be large enough and then the right index to be even larger. For backwards Cauchy we choose the right index first and then the left index. \\
A sequence $\{ x_n \} \subseteq \os$ is \emph{forwards admissible} if for all $\eps>0$ there is a natural number $N(\eps)$ such that for all $n >N(\eps)$ there is a natural number $K(n,\eps) > n$ such that for all $k>K(n,\eps)$, 
\begin{equation}\label{eq22}d(x_n, x_k)<\eps. \end{equation}
Refer to Figure \ref{AdmissibleFig} for an illustration of the indices. 
A backwards admissible sequence is defined similarly (choosing the right index first). \\
When we leave out the adjective forwards or backwards for an admissible sequence we shall always mean a forwards admissible sequence. 
\end{defn}

\begin{figure}[ht]
\begin{center}
\includegraphics{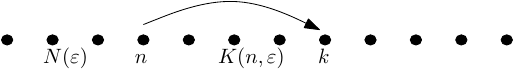}
\caption{\label{AdmissibleFig}A forwards admissible sequence, the indices are increasing to the right and 
the arrow shows the direction in which the distance is small.}
\end{center}
\end{figure}

Note that if $d$ is a symmetric metric then the Cauchy and admissible definitions are equivalent.

\begin{example}\label{NonSymEx}
Consider $X=[0,1]$ where $d(x,x') = x-x'$ 
if $x>x'$ and $d(x,x') =1$ if $x<x'$. The sequence $x_n = \left\{ \begin{array}{ll} \frac{1}{n} & n \text{ odd}\\[0.2 cm] \frac{1}{2n} & n \text{ even} \end{array}\right.$ is a forwards admissible sequence that is not forwards Cauchy. The point $x =0$ is its CFL. 
\end{example}

\begin{remark}\label{NKdefn}
We use the convention that $N(\eps), K(n,\eps)$ denote the smallest integers with the required property.
Clearly, the indexes $N$ and $K$ also depend on the sequence $\xi$ and when emphasizing this dependence will write $N(\xi, \eps)$ etc. 
\end{remark}

Unlike in a symmetric metric space, it is false that a convergent sequence is a Cauchy sequence. Even more disturbingly, a sequence admitting a CFL may not have an admissible subsequence! 

\begin{example}\label{ConvNoCauchy} 
Let $X$ be the space from Example \ref{NonSymEx}, we glue countably many different copies of $X$ along the point $0$:
\[ Y = X \times \N / (0,i) \sim (0,j) \text{ for } i,j \in \N. \]
and extend the metric by declaring $d((x,i), (x',j)) = 1$ if $x\neq 0$ and $x'\neq0$ and $i \neq j$. Consider the sequence $y_k = (\frac{1}{k}, k)$. Then $y_k$ has a CFL, it is the equivalence class of the point $(0,1)$. But for each $n \neq k$: $d(y_n, y_k) =1$ so it admits no admissible subsequence. 
\end{example}

However, we do have the following properties:

\begin{prop}\label{ssProp}
If $\{x_n\}$ is a forwards admissible sequence and $\{x_{n_j}\}$ is an infinite subsequence then
\begin{enumerate}
\item $\{ x_{n_j} \}$ is admissible. 
\item if $x$ is a forward limit of $\{x_n\}$ then it is a forward limit of $\{x_{n_j}\}$.
\item if $x$ is a forward limit of $\{x_{n_j}\}$ then it is a forward limit of $\{x_n\}$.
\item $x = \fl{n} x_n$ if and only if $x=\fl{j} x_{n_j}$. 
\end{enumerate}
\end{prop}
\begin{proof}
Items (1) and (2) follow from the definitions. To prove (3) suppose $\lim_{j \to \infty} d(x_{n_j}, x) = 0$. Then given $\eps>0$ let $N(\eps)$ be the constant from Definition \ref{add_seq_defn} let $n>N(\eps)$ and let $K(n, \eps)$ be the constant from the same definition then for $n_j>K(n,\eps)$ we have:
\[ d(x_n,x) \leq d(x_n, x_{n_j}) + d(x_{n_j}, x) \leq d(x_{n_j},x) + \eps \] which limits to 0. 
Item (4) follows from (2) and (3). 
\end{proof}

\begin{obs}\label{obs1} In keeping with the conventions of Remark \ref{NKdefn}, we observe that if $\xi$ is forwards admissible and $\eps' \leq \eps$ then
\begin{enumerate}
\item $N(\eps) \leq N(\eps')$. 
\item $K(n,\eps) \leq K(n,\eps')$
\item $n' \geq K(n,\eps)$ implies $K(n,\eps) \leq K(n',\eps)$
\end{enumerate}
\end{obs}

\begin{prop}\label{addConCauchy}
Every forwards admissible sequence $\xi = \{ x_n\}$ has a subsequence $\{x_{n_i}\}$ which is forwards Cauchy. Moreover we can choose this subsequence so that for all $i<j$ we have \[ d(x_{n_i},x_{n_j})< \frac{1}{2^i}\]
\end{prop}
\begin{proof}
For convenience let us denote $x(n) = x_n$. Since $\{x(n)\}$ is admissible, then for $\eps>0$ let $N(\eps)$ and for $n>N(\eps)$ let $K(n,\eps)$ be the constants from Definition \ref{add_seq_defn}.
Then the Cauchy subsequence will be given recursively by $n_1 = N(\frac{1}{2})$ and 
\[n_{j+1} = \max \left\{ N \left( \frac{1}{2^{j+1}} \right)+1 , K \left( n_j, \frac{1}{2^{j+1}} \right) \right\}\] 
For all $j>i$ we have $n_i>N(\frac{1}{2^i})$ and $n_j \geq K( n_{j-1}, \frac{1}{2^j})$ and by \ref{obs1}(3) $n_j \geq K( n_i, \frac{1}{2^i})$ hence $d(x(n_i), x(n_j))< \frac{1}{2^i}$. \end{proof}

We define the limiting distance between two admissible sequences. 

\begin{lemma}\label{dist_limit}
Let $\xi = \{x_n\}, \eta=\{y_n\}$ be forwards admissible then $\displaystyle c(\xi,\eta) = \lim_{n \to \infty} \lim_{k \to \infty} d(x_n,y_k)$. That is, one of the two options hold: 
\begin{enumerate}
\item\label{case1} For all $r>0$ there is an $N(r) \in \NN$ so that for all $n>N(r)$ there is a $K(n,r) \in \NN$ such that, \[ d(x_n, y_k) > r \quad \text{for all} \quad k>K(n,r)\]
In this case we write: $c(\xi,\eta) = \infty$.
\item\label{case2} There is a number $c(\xi,\eta)\geq 0$ such that for all $\eps>0$ there is an $N(\eps) \in \NN$ so that for all $n>N(\eps)$ there is a $K(n,\eps) \in \NN$ such that, \[| d(x_n, y_k) - c(\xi,\eta)| < \eps \quad \text{for all} \quad k>K(n,\eps)\]
\end{enumerate}
\end{lemma}

We shall need the following definition and proposition to prove this lemma. 

\begin{defn}\label{Defn:AlmostMonotonicallyDec}
A sequence $\{ r_n\}_{n=1}^{\infty}$ in $\RR$ is \emph{almost monotonically decreasing} if for every $\eps>0$ there is a natural $N(\eps)$ such that for all $n>N(\eps)$ there is a natural number $K(n,\eps)$ so that for all $k>K(n,\eps)$ 
\[ r_{k} \leq r_n+\eps \] 
$\{ r_n\}_{n=1}^{\infty}$ is almost monotonically increasing if $\{- r_n\}_{n=1}^{\infty}$ is almost monotonically decreasing.
\end{defn}
\begin{prop}\label{Prop:BoundedAndAlmostMonDec}
If $\{ r_i\}$ is almost monotonically decreasing and bounded below or almost monotonically increasing and bounded above then it converges. 
\end{prop}

\begin{proof} We shall prove the almost monotonically decreasing case. 
Since $\{ r_i\}$ is bounded, there is a subsequence $\{ r_{i_j}\}$ converging to some number $R$. 
We show that $R$ is in fact the limit of $\{ r_i\}$. Let $\eps>0$, and $M=M(\eps)$ be such that for $j \geq M$ we have $|r_{i_j} - R| < \eps$. 
Let $K = K(i_M,\eps)$ be the constant from Definition \ref{Defn:AlmostMonotonicallyDec}. So for any $k>K$ 
\begin{equation}\label{eq2321} r_k \leq r_{i_M} + \eps < R+2\eps \end{equation}
For the other inequality, let $K' = K(k,\eps)$ be the constant from the almost monotonically decreasing 
definition \ref{Defn:AlmostMonotonicallyDec}, and choose $s$ large enough so that $i_s>K'$ and $s>M$ then
\begin{equation}\label{eq24} R - \eps < r_{i_s} \leq r_k +\eps \end{equation}
From equations \ref{eq2321} and \ref{eq24} we get $|r_k-R| < 2\eps$.
\end{proof}

\begin{proof}[Proof of Lemma \ref{dist_limit}]
We must show that for every pair of admissible sequences $\xi = \{x_n\}$, $\eta = \{y_n\}$ 
satisfy either (1) or (2). Let $\eps>0$ fix $n \in \NN$ and construct the sequence $\al(n) = \{a_k\}_{k=n}^\infty$ in $\RR$ by defining $a_k = d(x_n,y_k)$ for $n \leq k \in \NN$. 
For a fixed $n$, the sequence $\al(n)$ is non-negative and almost monotonically decreasing. This follows from the triangle inequality and because $\eta$ is admissible. 
By Proposition \ref{Prop:BoundedAndAlmostMonDec} $\al(n)= \{a_k\}_{k=n}^\infty$ converges to a limit $c_n$.

We now show that $\{c_n\}_{n=1}^\infty$ is almost monotonically increasing.
Since $\lim_{j \to \infty} d(x_n, y_j) = c_n$ 
then for each $n$ let $J(n,\eps)$ be an index so that for all $j>J(n,\eps)$ \begin{equation}\label{eq25} |d(x_n,y_j) - c_n|< \eps. \end{equation} 
We can assume that $J(n,\eps)$ is monotonically increasing with $n$. 
Let $N:= N_{\ref{add_seq_defn}}(\xi,\eps)$ be the constant from Definition \ref{add_seq_defn} so that for all $n>N$, there is a $K_{\ref{add_seq_defn}}(\xi,n,\eps)$, so that $k>K_{\ref{add_seq_defn}}(\xi,n,\eps)$ we have $d(x_{n}, x_k)< \eps$. 
We take $n>N$, $j>K(n,\eps)$ and $t>\max\{j, J(j,\eps)\}$ (then $t>J(n,\eps)$ ), then we have:
\begin{align}
c_j & \geq d(x_j,y_t) - \eps \label{eqA} \\[0.2 cm]
&\geq d(x_n,y_t) - d(x_n,x_j) - \eps \label{eqB} \\[0.2 cm]
& \geq d(x_n,y_t) - 2 \eps \label{eqC} \\[0.2 cm]
&\geq c_n - 3 \eps \label{eqD}
\end{align}

Therefore $\{c_n \}$ is almost monotonically increasing. So either 
\begin{itemize}
\item $\{c_n\}$ is bounded above and hence converges to a limit $c$. This implies case \ref{case2} of the statement, or,
\item $\{c_n\}$ is unbounded and so $\xi, \eta$ satisfy case \ref{case1} of the statement. \qedhere
\end{itemize}
\end{proof}

\begin{obs}\label{cTriangleInequality}
The function $c( \cdot , \cdot )$ satisfies the triangle inequality. 
\end{obs}
\begin{proof}
Let $\xi = \{x_n\}, \eta=\{y_n\}, \zeta = \{z_n\}$ be admissible sequences, then for indexes $n<k<m$ we have:
\[ d(x_n, z_m) \leq d(x_n, y_k) + d(y_k, z_m) \] 
The triangle inequality for $c(\xi,\zeta)$ follows. 
\end{proof}

\begin{defn}\label{interlace}
Let $\xi = \{x_n\}$ and $\eta = \{y_n\}$ be sequences in $X$. We denote their interlace sequence 
by $j(\xi, \eta) = \zeta = \{ z_n \}$ which is given by: 
\[ z_n = \left\{ \begin{array}{ll} x_{\frac{n+1}{2}} & n \text{ odd } 
\\ y_{\frac{n}{2}} & n \text{ even } \end{array} \right. \] 
We call the admissible sequences $\xi$ and $\eta$ \emph{neighbors}, 
if their interlace is admissible. 
\end{defn}

\begin{lemma}\label{equiv_relation}
Let $\xi = \{x_n\}$ and $\eta = \{y_n\}$ be forwards admissible sequences then $\xi$
and $\eta$ are neighbors iff \[c(\xi,\eta) = 0 \text{ and } c( \eta,\xi)=0\]
\end{lemma}
\begin{proof}
Suppose the interlace $j(\xi,\eta) = \zeta = \{z_n\}$ is admissible. For any $n$, $x_n = z_{2n-1}$ and $y_n = z_{2n}$. Therefore, for large $n<k$
\[ d(x_n,y_k) = d(z_{2n-1},z_{2k})\]
which is small provided $2n-1 >N_{\ref{add_seq_defn}}(\zeta,\eps)$ and $2k> K_{\ref{add_seq_defn}}(\zeta, 2n-1, \eps)$. This proves that $c(\xi,\eta) = 0$ and similarly $c(\eta,\xi) =0$. 

For the other deduction assume that both sequences are forwards admissible and that 
both limits are $0$.
We must find for all $\eps>0$ a natural number $N=N(\eps)$ and for all $n>N$ a natural number $K = K(n,\eps)$ so that for all $n>N$ and $k>K$
\begin{equation}\label{eq33} d(z_n,z_k)<\eps \end{equation}
We now have four cases:
If both $n,k$ are odd then inequality \ref{eq33} follows from the admissibility of $\xi$. 
If both $n,k$ are even then inequality \ref{eq33} follows from the admissibility of $\eta$. 
If $n$ is odd and $k$ even 
then inequality \ref{eq33} follows from $c(\xi,\eta)=0$.
If $n$ is even and $k$ odd 
then inequality \ref{eq33} follows from $c(\eta,\xi)=0$. 
\end{proof}

\begin{prop}\label{c_well_defined_on_neighbors} Suppose $\xi,\xi'$ and $\eta, \eta'$ are neighbors respectively. Then, 
\[ c(\xi, \eta) = c(\xi', \eta'). \] 
\end{prop}

\begin{proof}
If $\xi, \xi'$ are neighbors and $\eta,\eta'$ are neighbors then by Observation \ref{cTriangleInequality} and Lemma \ref{equiv_relation} we have $c(\xi, \eta) \leq c(\xi,\xi') + c(\xi', \eta') + c(\eta',\eta) = c(\xi',\eta')$. 
\end{proof}

\begin{remark}\label{CFLnotnbrs}
Notice that if $\xi = \{x_n\}_{n=1}^\infty$ and 
$\eta = \{y_n\}_{n=1}^\infty$ are 
admissible sequences with the same closest forward limit then they are not necessarily neighbors. 
For example, consider $X$ in Example \ref{NonSymEx} and the sequences $x_n = 0$ and $y_n = \frac{1}{n}$ for all $n$. 
\end{remark}

\begin{remark}
By Lemma \ref{equiv_relation} and the triangle inequality we have that being 
neighbors is an equivalence relation. However, as noted in Remark \ref{CFLnotnbrs} this is not the equivalence relation we should use to define the forwards completion. 
\end{remark}

\begin{notation}
For $x \in X$ we denote the constant sequence $\{x\}_{n=1}^\infty$ by $\bold{x}$.
\end{notation}

\begin{defn}\label{defnEquiv}[Equivalent sequences, realization in $X$]
We call the forwards admissible sequences $\xi$ and $\eta$ equivalent if either they are neighbors or they have the same CFL. This is an equivalence relation (using Proposition \ref{ssProp}). \\
If an equivalence class $\al$ has a representative $\xi\in\al$ with a CFL in $X$ denoted $x$, then the constant sequence $\bold{x}$ also belongs to $\al$ and we say that $\al$ is realized by $x$ in $X$. 
\end{defn}

\begin{obs}\label{obs2}
\begin{enumerate}
\item There exists some representative $\xi \in \al$ such that $x = \overrightarrow{\lim} \xi$ if and only if for all $\xi \in \al$, $x = \overrightarrow{\lim} \xi$.
\item If $\al$ is not realized then for all $\xi, \xi' \in \al$, $\xi$ and $\xi'$ are neighbors. 
\end{enumerate}
\end{obs}

\begin{defn}\label{defHat}
Let $\hat X$ be the quotient set of forwards admissible sequences in $X$ by the equivalence relation in Definition \ref{defnEquiv}. 
We define the function on the set $\hat X$ 
\[ \begin{array}{l}
\hat d: \hat X \times \hat X \to \RR\cup \{\infty\}\\
\hat d([\xi],[\eta]) = \inf \{ c(\xi',\eta') \mid \xi\sim \xi' \text{ and } \eta \sim \eta' \}
\end{array}\]
\end{defn}

\begin{example}
It is possible that for $x, y \in X$, $d(x,y) > \hat d([\bold{x}], [\bold{y}])$.
Let $X = [0,1]$ and suppose $d(x,y) = x-y$ for $x>y$, $d(x,y) = 1$ for $0 \neq x<y$ and $d(0,y) = 2$ for all $y$. Then $d(0,1) = 2 > \lim_{k \to \infty} d(\frac{1}{k}, 1) = 1 = \hat d( \bold{0}, \bold{1})$. We therefore define the notion of a forwards continuous metric space. 
\end{example}


\begin{defn}\label{contDefn}
An asymmetric metric $d$ is \emph{forwards continuous} if for 
every forwards admissible sequence 
$\{x_m\}_{m=1}^\infty$ that admits a CFL $x$, and for any $y \in X$ we have 
\begin{align}
d(x, y) & = \lim_{m \to \infty} d( x_m,y ) \quad \text{ and } \quad \label{continuityEq1} \\
d(y,x) & = \lim_{m \to \infty} d(y, x_m ) \label{continuityEq2}
\end{align}
Similarly it is \emph{backwards continuous} if every backwards admissible sequence $\{x_m\}_{m=1}^\infty$ that admits a CBL $x$, and for any $y \in X$, equations (\ref{continuityEq1}) and (\ref{continuityEq2}) hold.
\end{defn}

\begin{remark}
The triangle inequality automatically implies that $d(y, \cdot)$ is lower semi-continuous and $d(\cdot, y)$ is upper semi-continuous. Explicitly, if 
$\displaystyle \lim_{m \to \infty} x_m = x$ then 
$d(x_m, y) \leq d(x,y) + \eps$ for large $m$ and 
$d(y, x_m) \geq d(y,x) - \eps$ for large $m$. 
\end{remark}

\begin{remark}\label{osFrwrdsCont}
We will show that Outer Space $\os$ with the Lipschitz metric is both forwards and backwards continuous (see Proposition \ref{conditions}). 
\end{remark}

\begin{example}\label{notMinimizing}
Even when $X$ forwards continuous, the infimum in $\hat d$ may not be realized.
Consider the space $X = \{x^i_n\}_{i,n=1}^\infty \coprod \{y_m\}_{m=1}^\infty \coprod \{0\}$ (see Figure \ref{infNotMin}). 
We define an asymmetric metric as follows:
\[ d(x^i_n, 0) = \frac{1}{n} \]
and 
\[\begin{array}{ll}
d(x^{i}_n, y_m) =
\left\{ \begin{array}{ll}
2 & n \geq m \\
1+ 1/i & n<m
\end{array} \right. , 
&
d(0,y_m) = 2 
\end{array} \]
We also have,
\[\begin{array}{ll}
d(x^i_n, x^i_k) = 
\left\{ \begin{array}{ll}
1/n-1/k & n < k \\
1 & n>k
\end{array} \right. , 
& 
d(0, x^i_k) = 1
\end{array}\]
We define 
\[d(y_n, y_m) = 
\left\{ \begin{array}{ll}
1/n-1/m & n < m \\
1 & n>m
\end{array} \right. 
\]
And the rest of the distances are equal to $1$, e.g. $d(y_m, x^i_k)$, $d(x^i_n, x^j_m)$ for $i \neq j$. One can directly check that the directional triangle inequality holds and therefore, $X$ is an asymmetric metric space.

\begin{figure}[ht]
\begin{center}
\includegraphics{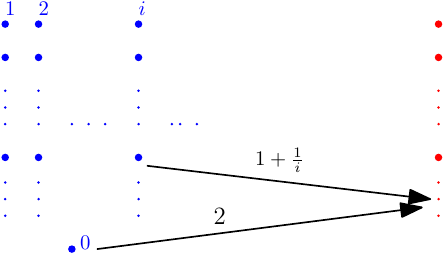}
\caption{\label{infNotMin}An example where the infimum in the definition of $\hat d$ is not realized. The sequences $\{x^i_n\}_{n=1}^\infty$ are the columns on the left and $\{y_m\}$ is on the right.}
\end{center}
\end{figure}

It is elementary to check that $X$ is continuous. The sequence $\eta = \{y_n\}$ is admissible as well as $\xi_i = \{x^i_n\}_{n=1}^\infty$ for all $i \geq 1$. Moreover, for each $i$, $0 = \fl{n} x^i_n$. Thus, $\xi_i \in [\bold{0}]$ for all $i$. We have
\[ c(\xi_i, \eta) = \lim_{n \to \infty} \lim_{m \to \infty} d(x_n^i, y_m) = 1 + \frac{1}{i}\]
Therefore $\hat{d}([\bold{0}], [\eta]) \leq 1$. However, if $\xi$ is an admissible sequence whose CFL is $0$ then either it eventually becomes the constant sequence $\bold{0}$ or there exists an $i$ so that it eventually becomes a subsequence of $\{x^i_n\}_{n=1}^\infty$. Thus, $c(\xi, \eta) > 1$ for any $\xi \in [\bold{0}]$. Hence $\hat d$ is not realized by any sequence. 
\end{example}

\begin{lemma}\label{infIsMin}
If $(X,d)$ is a forwards continuous metric space then
\begin{enumerate}
\item For all $\al, \beta \in \hat X$, if $\al$ is not realized then the infimum in $\hat d(\al, \beta)$ is a minimum, and 
\item If both $\al, \beta$ are not realized then 
\[ \hat d(\al,\beta) = c(\xi, \eta) \text{ for all } \xi \in \al, \eta \in \beta \]
\item If $\beta$ is realized by $y$ in $X$ then 
\[ \hat d(\al, \beta) = \inf \{ c(\xi, \bold{y}) \mid \xi \in \alpha\} \]
If $\al$ is not realized then for all $\xi \in \al$, $\hat d(\al,\beta) = c(\xi, \bold{y})$. 
\item If both $\al, \beta \in \hat X$ are realized by $x,y \in X$ then \[\hat d(\al,\beta) = d(x,y).\]
\end{enumerate}
\end{lemma}

\begin{proof}

To prove (3), for any admissible $\xi = \{x_m\} \in \al$ and any $\eta = \{y_k\} \in \beta$:
\[ \begin{array}{ll}
\displaystyle c(\xi,\eta) = \lim_{m \to \infty} \lim_{k \to \infty} d(x_m, y_k) & \displaystyle \geq 
\lim_{m \to \infty} \lim_{k \to \infty} d(x_m, y) - d(y_k , y) \\ 
\displaystyle & = \displaystyle
\lim_{m \to \infty} d(x_m,y) \\ 
\displaystyle & = c(\xi, \bold{y}) 
\end{array}\]
Moreover, if $\al$ is not realize then by Observation \ref{obs2} for all $\xi,\xi'\in\al$, $\xi$ and $\xi'$ are neighbors so by Proposition \ref{c_well_defined_on_neighbors}, $c(\xi,\bold{y}) = c(\xi',\bold{y})$. 
To prove (4), we have that for all $\xi = \{x_n\} \in \al$, $\fl{n} x_n = x$ and by continuity, $d(x,y) = \lim_{m \to \infty} d(x_m,y) = c(\xi,\bold{y})$. By (3), $d(x,y) = \hat d( \al,\beta)$.
To prove (2), if both $\al$, $\beta$ are not realized then for all $\xi,\xi' \in \al$ and $\eta, \eta' \in \beta$ we have $c(\xi, \eta) = c(\xi',\eta')$ by Observation \ref{obs2} and by Proposition \ref{c_well_defined_on_neighbors}. Thus, the infimum is taken over a set of one element. 
To prove (1), the case where $\beta$ is not realized follows from (2) and if $\beta$ is realized this follows from (3). 
\end{proof}

\begin{defn}\label{miniDefn}
Let $\al \in \hat X$ if $\xi \in \al$ is an admissible sequence in $X$ such that 
\begin{equation*} 
\hat d(\al, \beta) = \inf \{c(\xi, \eta) \mid \eta \in \beta\} 
\end{equation*}
for all $\beta \in \hat X$ then $\xi$ is called a good representative of $\al$. If each $\al \in \hat X$ admits a good representative then we shall say that $X$ is \emph{minimizing}. 
\end{defn}

\begin{remark}
Notice that if $\al$ is realized by $x$ then $[\bold{x}]$ is not necessarily a good representative (see Example \ref{notMinimizing}). 
\end{remark}

\begin{prop}\label{realizationOfDistance}
If $(X,d)$ is forwards continuous and minimizing then for $\xi$ a good representative of $\al$ , 
\[\hat d(\al, \beta) = c(\xi, \eta) \text{ for all } \eta \in \beta \]
\end{prop}
\begin{proof}
If $\beta$ is not realized then for all $\eta' \in \beta$,
\[ \hat d(\al,\beta) = \inf \{ c(\xi, \eta) \mid \eta \in \beta\} = c(\xi, \eta') \]
the last inequality follows since all elements of $\beta$ are neighbors. If $\beta$ is realized by $y$ then by Lemma \ref{infIsMin},
\[ \hat d(\al, \beta) = \inf \{ c(\xi', [\bold y]) \mid \xi' \in \al \} = c(\xi, [\bold y]) \]
Let $\xi = \{x_n\}$ and $\eta = \{y_m\}$ then by continuity,
\[ c(\xi, \eta) = \lim_{n \to \infty} \lim_{m \to \infty} d(x_n, y_m) = c(\xi, [\bold{y}]) \qedhere \]
\end{proof}

\begin{lemma}\label{triDHatLemma}
If $(X,d)$ is forwards continuous and minimizing then the function $\hat d$ satisfies the triangle inequality.
\end{lemma}
\begin{proof}
Let $\al,\beta,\gamma \in \hat X$ and let $\xi,\eta,\nu$ be good representatives of these classes respectively. Then,
$\hat d( \al , \beta ) = c( \xi,\eta )$, 
$\hat d( \al , \gamma) = c(\xi , \nu)$ and $\hat d(\gamma ,\beta ) = c(\nu , \eta )$. The triangle inequality of $\hat d$ now follows from Observation (\ref{cTriangleInequality}). 
\end{proof}

\begin{defn}\label{defnCompl}
Let $(X,d)$ be a forwards complete, minimizing, asymmetric metric space, then $(\hat X, \hat d)$ from Definition \ref{defHat} is an asymmetric metric space called the forwards completion of $(X,d)$.
\end{defn}

\begin{remark}
One way in which $\hat d$ may differ from $d$ is the separation axioms it satisfies. 
The function $d$ might satisfy \[ d(x,y) = 0 \implies x=y\] even while $\hat d$ may not. 
This in fact occurs in $\hat \os$ (see the proof of  Proposition \ref{isomExtend}). 
\end{remark}

We will need the following technical lemma. 

\begin{lemma}\label{convenientSubsequences}
If $\{\al_n\}$ is a sequence of admissible sequences so that $c(\al_n, \al_{n+1}) = c_{n,n+1}$ then for each $n$ there is a subsequence $\al_n'$ of $\al_n$ with $\al_n' = \{x_{n,k}\}_{k=1}^\infty$ so that for all $n$ and $k<k'$ we have
\begin{enumerate}
\item $d(x_{n,k}, x_{n,k'})\leq \frac{1}{2^k}$.
\item $d(x_{n,k}, x_{n+1,k'}) \leq c_{n,n+1} + \frac{1}{2^{nk}}$.
\end{enumerate} 
\end{lemma}

\begin{proof}
\textbf{Step 1:}
Using Proposition \ref{addConCauchy} we first extract subsequences $\al'_n$ of $\al_n$ so that property (1) holds. Note that $c(\al'_n, \al'_j) = c(\al_n, \al_j)$ for all $n,j$. 

\textbf{Step 2:}
Let $\al'_n = \{x_{n,k}\}$ be the subsequences extracted in Step 1. We now recursively extract subsequences of $\al_n'$ that satisfy Property (2) in the statement of the Lemma. Note that every subsequence of $\al_n'$ satisfies Property (1). 

Since $c(\al'_1, \al'_2) = c_{1,2}$ then for all $i \geq 1$ there exists a $K(i)$ so that for all $k \geq K(i)$ there exists a $J(k,i)$ satisfying $|d(x_{1,k}, x_{2,j}) - c_{1,2}| < \frac{1}{2^i}$. We take the subsequence $x_{1,i}' = x_{1,K(i)}$ of $\al_1'$, and (temporarily) take the subsequence of $\al'_2$ with indices $J(K(i), i)$. Thus for all $i$, the new sequences which we denote by $\{x_{1,i}'\}, \{x_{2,i}'\}$ satisfy \[ |d(x_{1,i}', x_{2,i}')-c_{1,2}|<\frac{1}{2^i}.\]
Notice that this inequality still holds for any subsequence of $x_{2,i}'$.

We now suppose that we have extracted subsequences of $\al'_n$ for $n \leq N$ that satisfy property (2) for each $n \leq N-1$. In the induction step we extract a subsequence of $\al'_{N}$ and a (temporary) subsequence of $\al'_{N+1}$ that satisfy property (2) for $n=N$. \\
Since $c_{N,N+1} = c(\al'_N, \al'_{N+1})$, then for each $i$ we will let $\eps = \frac{1}{2^{iN}}$ and thus there exists a natural number $K(i)$ so that for all $k>K(i)$ there exists a natural number $J(k,i)$ satisfying for all $j>J(k,i)$, 
\[ |d(x_{N,k}, x_{N+1,j})-c_{N,N+1}|< \frac{1}{2^{Ni}}.\]
We take the subsequence $x_{N, K(i)}$ of $\al'_{N}$ and the subsequence of $\al'_{N+1}$ whose indices are $J(K(i), i)$. We denote the new sequences by $\{x'_{N, i}\}$ and $\{x'_{N+1,i}\}$. Then for all $i$ we have: 
$x'_{N,i} = x_{N,K(i)}$ and for $j \geq i$ we have $x'_{N+1, j} = x_{N+1, J(K(j),j)}$ and $J(K(j), j) \geq J(K(i), i)$, thus 
\[d(x'_{N, i}, x'_{N+1, j})< \frac{1}{2^{Ni}} \text{ for } j\geq i.\] 
The subsequence of $\al'_{N+1}$ will be changed to a new subsequence in the next step but this process does not ruin property (2) for $n=N$. 
\end{proof}

\begin{definition}
An asymmetric space $(Y,\rho)$ is
\emph{forwards complete} if every forwards admissible 
sequence $\{ y_n \}$ in $Y$ has a closest forward limit. 
\end{definition}

\begin{lemma}\label{bigLemma}
If $(X,d)$ is a forwards continuous, minimizing, asymmetric metric space then its forward completion $(\hat X, \hat d)$ is forwards complete. 
\end{lemma}
\begin{proof}
Let $\{\al_n\} \subset \hat X$ be an admissible sequence, we will first show that it has a forwards limit $\al \in \hat X$. By Proposition \ref{ssProp} and Proposition \ref{addConCauchy}, we may switch $\{\al_n\}$ by a subsequence so that 
\begin{equation}\label{eq55} \hat d (\al_n,\al_m)< \frac{1}{2^n} \text{ for } 1 \leq n < m. \end{equation}
By Proposition \ref{realizationOfDistance} we may choose 
sequences $\xi_n$ so that for all $n$:
\[ c(\xi_n, \xi_{n+1}) = \hat d(\al_n, \al_{n+1}) = \frac{1}{2^n}.\]
We can now use Proposition \ref{convenientSubsequences} to extract subsequences $\xi'_n$ of $\xi_n$ so that for $\xi'_n = \{x_{n,k}\}_{k=1}^\infty$ we have for all $n$ and for all $k\leq j$:
\begin{enumerate}
\item $d(x_{n,k}, x_{n,j}) < \frac{1}{2^k}$, and
\item $d(x_{n,k}, x_{n+1,j})< \frac{1}{2^n} + \frac{1}{2^{nk}}$
\end{enumerate}
We now let $\xi = \{x_{n,n}\}$ and note that it is a Cauchy sequence. We also see that for $k>n$ we have:
\[ d(x_{n,k}, x_{k,k}) \leq \sum_{j=n}^{k-1} d(x_{j,k}, x_{j+1, k}) < \sum_{j=1}^{k-1} \frac{1}{2^j} + \frac{1}{2^{jk}} \leq \frac{1}{2^{n-2}}\] 
We therefore see that $\lim_{n \to \infty} c(\xi_n, \xi) =0$ hence denoting $\al = [\xi]$ we have that $\al$ is a forward limit of $\al_n$. 

We claim that it is the closest forward limit. Indeed, let $\zeta = \{z_m\}$ so that $[\zeta]$ is a forward limit of $\al_n$. The same sequences 
$\{\xi_n\}$ satisfy $c(\xi_n,\zeta) = \hat d(\al_n, [\zeta]) = 0$. Thus, for all $\eps$ there exists an $N(\eps)$ so that for all $n>N(\eps)$ there is an $M(n, \eps)$ such that for $m\geq M(n,\eps)$ there is a $K(m,n,\eps)$ satisfying
\[ d(x_{n,m}, z_k)< \eps \text{ for } k>K(m,n,\eps) \]
This implies for all $n>N(\eps)$ and $k> K(M(n,\eps),n, \eps)$ that
\[ d(x_{n,n}, z_k) \leq d(x_{n,n}, x_{n,M(n,\eps)}) + d(x_{n,M(n, \eps)}, z_k) \leq \frac{1}{2^{n-1}}+\eps \]
Hence $c(\xi, \zeta) = 0$. This implies that $\al = [\xi]$ is the CFL of $\{\al_n\}$. 

\end{proof}

\begin{prop}\label{hatXsatisfies11}
Under the assumptions of Lemma \ref{bigLemma}, $\hat X$ need not necessarily be forwards continuous but it does satisfy Equation (\ref{continuityEq1}) of the definition of forwards continuous. 
\end{prop}
\begin{proof}
The example in Remark \ref{remarkBadExample} shows that $\hat X$ need not be forwards continuous. We will show that it satisfies Equation (\ref{continuityEq1}). Let $\al \in \hat X$ be the CFL of the admissible sequence $\{\al_n\} \subset \hat X$. 
Let $\beta \in \hat X$ be any point, we wish to show that $\hat d(\al, \beta) = \lim_{n \to \infty} \hat d( \al_n, \beta)$. 
As before we find sequences $\xi_n = \{x_{n,k}\}, \eta = \{y_m\} \subset X$ so that $\eta$ is a good representative of $\beta$ and $\xi_n$ are good representatives of $\al_n$. We extract subsequences satisfying items (1)+(2) as before, and denote by $\xi = \{x_{n,n}\}$. 
Then since $\al$ and $[\xi]$ are both CFLs, $\xi \in \al$. For $\eps>0$ let $N(\eps)$ be so that for all $n>N(\eps)$, $c(\xi_n, \eta) = \hat d(\al_n, \beta)$ is within $\eps$ of $L = \lim_{n \to \infty} \hat d(\al_n, \beta)$. Let $K(n,\eps)$ and for $k>K(n,\eps)$ let $M(n,k,\eps)$ be as in Definition \ref{dist_limit}, then take $k= \max\{n,K(n,\eps)\}$ and $m > M(n,k,\eps)$ then
\[ d(x_{n,n}, y_m) \leq d(x_{n,n}, x_{n,k}) + \hat d(x_{n,k}, y_m) \leq \frac{1}{2^n} + c(\xi_n,\eta) + \eps \leq L + 3 \eps\] for large $n$. 
This implies $\hat d(\al, \beta) \leq c(\xi, \eta) \leq L$. 
Moreover, $\hat d(\al_n, \beta) \leq \hat d(\al_n , \al) + \hat d(\al, \beta)$ implies $L \leq \hat d(\al,\beta)$ which gives the equality we need.
\end{proof}

\begin{defn}\label{denseDefn}
If $X \subset Y$ and for all $y \in Y$, there exists a forwards admissible sequence $\{y_n\} \subset X$ such that $y = \fl{n} y_n$ then we will say that $X$ is forward dense in $Y$. 
\end{defn}

\begin{lemma}\label{embedding}
The map $\iota: X \hookrightarrow \hat X$ defined by $\iota(x) = [\bold x]$  is an isometric embedding and its image $\iota(X)$ is forward dense. 
\end{lemma}
\begin{proof}
The fact that $\iota$ is an isometric embedding follows from Lemma \ref{infIsMin}(4). Let $\al \in \hat X$ then there exists a forwards admissible sequence $\xi \subset X$ such that $\al = [\xi]$. If $\xi = \{x_n\}_{n=1}^\infty$ then for $\eps>0$ since $\xi$ is admissible, let $n>N_{\ref{add_seq_defn}}(\eps)$ and $k>K_{\ref{add_seq_defn}}(n,\eps)$ 
\[ c(\bold{x_n}, \xi) = \lim_{k \to \infty} d(x_n, x_k) <\eps. \]
Thus, $\hat d(\iota(x_n), \al)< \eps$. Hence $\al$ is a forward limit of the sequence $\{ \iota(x_n)\}$. Moreover, if $\beta \in \hat X$ so that $\lim_{n \to \infty} \hat d(\iota(x_n), \beta) = 0$. Then there exists an admissible sequence in $X$, $\eta = \{y_m\} \in \beta$, so that for all $\eps$ there exists an $M(\eps)$ so that for all $n>M(\eps)$, $c(\iota(x_n),\eta)<\eps$. Thus, there exists a $J(n,\eps)$ so that
$d(x_n,y_j)<\eps$ for all $j>J(n,\eps)$. But this exactly means that $c(\xi, \eta)<\eps$ and hence $\hat d(\al,\beta) <\eps$. Thus $\al$ is the closest forward limit of $\{\iota(x_n)\}$. 
\end{proof}

The space $\hat X$ will not necessarily be forwards continuous (Definition \ref{contDefn}) as we shall later see in Remark \ref{remarkBadExample}. 
However, we need some continuity property of $\hat X$ to hold in order to extend an isometry of $X$ to $\hat X$. 
We weaken the notion of forwards continuous to continuity with respect to a subspace as follows.

\begin{defn}\label{semiForwardsContinuous}
Suppose $(Y,d)$ is an asymmetric metric space and $X \subset Y$. Suppose that for all $a,b \in Y$ and all forwards admissible sequences $\{a_n\}, \{b_n\} \subset X$ such that $\fl{n}a_n = a$ and $\fl{n}b_n=b$ we have,
\[ d(a,b) = \lim_{n \to \infty} \lim_{k \to \infty} d(a_n, b_k) = c(\{a_n\}, \{b_n\}) \]
then we say that $Y$ is semi forwards continuous with respect to $X$. 
\end{defn}

For $Y = \hat X$ we will say that $\hat X$ is semi forwards continuous with respect to $X$ when we mean with respect to $\text{Im}(\iota)$.

\begin{remark}\label{remarkBadExample}
$Y$ may be forwards continuous and minimizing but still $\hat Y$ may NOT be semi-forwards continuous with respect to $Y$. Indeed consider the space $X$ of Example \ref{notMinimizing}. Let $Y = X \coprod \{x^0_n\}_{n = 1}^\infty$ and extend the distance:
\[ d(x^0_n, x^0_m) = \left\{ \begin{array}{ll} 
\frac{1}{n}-\frac{1}{m} & n<m \\
1 & n>m
\end{array} \right. , \quad d(x^0_n,0) = \frac{1}{n} , \quad 
d(x^0_n, y_m) = \left\{ \begin{array}{ll} 
1 & n<m \\
2 & n>m
\end{array} \right.
\]
and the unmentioned new distances are equal to $1$. 
Then $Y$ is forwards continuous (since $X$ was) but now it is also minimizing since for $\eta = \{y_m\}$ there exists $\xi = \{x^0_n\}_{n=1}^{\infty} \in [\bold{0}]$ such that 
\[ \hat d([\bold{0}], [\eta]) = c(\xi,\eta) = 1.\]
However, $\hat Y$ is not semi forwards continuous with respect to $Y$ since for $\bold{0}$, $c(\bold{0}, \eta) = 2 > c(\xi, \eta) = 1$. So we have a sequence $a_n = [\bold{0}] \in \text{Im}(\iota)$ and $b_m = [\bold{y_m}] \in \text{Im}(\iota)$ with $\fl{n} a_n = [\bold{0}]$ and $\fl{n}b_n = [\eta]$ so that 
\[ 2 = \lim_{n \to \infty} \lim_{m \to \infty} \hat d(a_n, b_m) \neq \hat d([\bold{0}],[\eta]) =1\] contradicting the semi forwards continuous property. 
\end{remark}

It is not so clear how to check that $X$ is minimizing or that $\hat X$ is semi-forwards continuous with respect to $X$. The following Lemma gives a sufficient condition for both.

\begin{lemma}\label{keyLemma}
Suppose that $(X,d)$ is an asymmetric metric space so that: 
\begin{equation}\label{star}
\begin{array}{l}
\text{If } \{x_n\}_{n =1}^\infty \subset X \text{ is a forwards admissible sequence that has a} \\
\text{CFL } x \in X 
\text{ then } \{x_n\} \text{ is also backwards admissible and } x \text{ is }\\ \text{its CBL.} 
\end{array}\tag{*}
\end{equation}
If $X$ is forwards continuous then $X$ is minimizing and $\hat X$ is semi forwards continuous with respect to $X$. 
\end{lemma}
\begin{proof}
First notice that if $\xi, \xi'$ are forwards admissible sequences with the same CFL $x$ then they are neighbors. Indeed since $\lim_{k \to \infty} d(x_k',x) = 0$ then $\lim_{k \to \infty} d(x,x'_k) = 0$. Therefore, $d(x_n, x'_k) \leq d(x_n, x)+ d(x, x'_k)$ limits to $0$ as $n,k \to \infty$. 

Now, for every $\al \in \hat X$, and for every $\xi,\xi' \in \al$, $\xi$ and $\xi'$ are neighbors (whether or not $\al$ is realized). This implies by Lemma \ref{c_well_defined_on_neighbors} that
\[ \hat d(\al,\beta) = c(\xi,\eta) \quad \text{ for every } \xi\in \al, \eta\in\beta. \]
In particular $X$ is minimizing.

To show that $\hat X$ is semi forwards continuous with respect to $X$, let $\al,\beta \in \hat X$ and $\{\al_n\}, \{\beta_k\} \subset \text{Im}(\iota)$ forwards admissible sequences with $\al=\fl{n} \al_n$ and $\beta = \fl{n}\beta_n$. 
We have to show that $\hat d(\al, \beta) = \lim_{n \to \infty} \lim_{k \to \infty} \hat d(\al_n,\beta_k)$. 
There exist $a_n, b_n \in X$ such that $\iota(a_n) = \al_n$ and $\iota(b_n) = \beta_n$. 
Note that $\hat d(\al_n,\beta_k) = c(\bold{a_n}, \bold{b_k}) = d(a_n,b_k)$.
Moreover, $\fl{n} \al_n = \al$ implies that $\{a_n\}_{n=1}^\infty \in \al$ and similarly, $\{b_n\}_{n=1}^\infty \in \beta$. 
Therefore,
\[ \lim_{n \to \infty} \lim_{k \to \infty} \hat d(\al_n,\beta_k) = 
\lim_{n \to \infty} \lim_{k \to \infty} d(a_n,b_k) = c(\{a_n\}, \{b_n\}) = \hat d(\al,\beta).\qedhere \]
\end{proof}


\begin{lemma}\label{connection}
If $X$ is forwards continuous, and $\hat X$ is semi-forwards continuous with respect to $X$ then $X$ is minimizing.
\end{lemma}
\begin{proof}
We remark that we will not use the directed triangle inequality for $\hat X$, for which we assumed that $X$ was minimizing. 
By Lemma \ref{infIsMin} we know that $\hat d(\al,\beta)$ is a minimum for all $\al \in \hat X$ that are not realized. Now assume $\al$ is realized by $x$. 
We need to show that for all $\beta \in \hat X$, $\hat d(\al,\beta) = \min \{ c(\xi,\eta) \mid \xi \in \al, \eta \in \beta \}$. 
Let $\xi \in \al$ then $\xi = \{x_n\} \subset X$ and $\fl{n} x_n = x$ and let $\eta = \{y_n\} \subset \beta$. Then $\al_n = \iota(x_n) \in \text{Im}(\iota)$ satisfies $\fl{n} \al_n =\al$, and similarly for $\beta_n = \iota(y_n) \in \text{Im}(\iota)$, $\fl{n} \beta_n = \beta$. Note that by Lemma \ref{infIsMin}, $\hat d(\al_n, \beta_k) = c(\bold{x_n}, \bold{y_k}) = d(x_n, y_k)$. By semi forwards continuity,
\[ \hat d(\al, \beta) = \lim_{n \to \infty} \lim_{k \to \infty} \hat d(\al_n, \beta_k) = \lim_{n \to \infty} \lim_{k \to \infty} d(x_n, y_k) = c(\xi,\eta)\]
Thus $\xi$ and $\eta$ realize the distance for every $\xi \in \al$ and $\eta \in \beta$. 
\end{proof}

We now prove that we can extend an isometry of $X$ to $\hat X$. 

\begin{theorem}\label{extIsom}
Let $X$ be a forwards continuous asymmetric metric space and suppose $\hat X$ is semi forwards continuous with respect to $X$. 
Suppose $(Y,\rho)$ is a forwards complete asymmetric metric space and let
\[ f:(X,d) \into (Y,\rho)\] be an isometric embedding so that $Y$ is semi forwards continuous with respect to $\text{Im}(f)$. Then there exists a 
lift of $f$ to an isometric embedding $F:(\hat X, \hat d) \into (Y, \rho)$ so that $f = F \circ \iota$.
\end{theorem}
\begin{proof}
By Lemma \ref{connection}, $X$ is minimizing. 
Let $\al \in \hat X$ and let $\xi = \{x_k\} \subset X$ be an admissible sequence so that $\lim_{k \to \infty} \iota(x_k) = \al$. 
The sequence $\{f(x_k)\}$ is admissible in $Y$ and $Y$ is forwards complete, hence it has a CFL $y$. We define $F(\al) = y$. 

We now check that $F$ is well defined: If $\al \in \hat X - \iota(X)$ then for all $\xi, \xi' \subset X$ such that the CFL of both $\iota(\xi), \iota(\xi')$ is $\al$, we have $\xi,\xi' \in \al$. Since $\al$ is not realized, then $\xi,\xi'$ are neighbors. Therefore, $f(\xi), f(\xi')$ are neighbors and therefore, they have the same CFL in $Y$. Now for the other case, if $\al \in \iota(X)$ then let $\al = \iota(x)$ and $\xi = \{x_n\}$ and $\xi' = \{x_n'\}$ sequences in $X$ so that $\fl{n} x_n = \fl{n} x_n' = x$. Denote by $y,y'$ the CFLs of $f(\xi),f(\xi')$. Since $Y$ is semi forwards 
continuous with respect to $\text{Im}(f)$ then $\rho(y,y') = \lim_{n \to \infty} \lim_{k \to \infty} \rho( f(x_n), f(x_k'))$. Since $f$ is an isometry $\rho(f(x_n), f(x_k')) = d(x_n, x_k)$. Therefore $\rho(y,y') = c(\xi,\xi')$. Since $X$ is forwards continuous we have $c(\xi,\xi') = \lim_{n \to \infty} \lim_{k \to \infty} d(x_n, x_k') = \lim_{n \to \infty} d(x_n, x) = 0$. Thus, $\rho(y,y') = 0$ and similarly $\rho(y',y) = 0$. This concludes the argument that $F$ is well defined. 

Given $\al,\beta \in \hat X$ let $\xi = \{x_n\} \in \al$ and $\eta = \{x'_n\} \in \beta$ then in $\hat X$ we have $\hat d(\al,\beta) = c(\xi, \eta)$ since $\hat X$ is semi forwards continuous with respect to $X$. We denote by $F(\al)$ the CFL of $f(\xi)$ and by $F(\beta)$ the CFL of $f(\eta)$. 
Since $Y$ is semi forwards continuous with respect to $\text{Im} f$ we have that $\rho(F(\al), F(\beta)) = \lim_{n \to \infty} \lim_{k \to \infty} \rho(f(x_n), \rho(f(x'_k))$. Since $f$ is an isometric embedding,
\[ \hat d(\al, \beta) = \lim_{n \to \infty} \lim_{k \to \infty} d(x_n, x'_k) = \lim_{n \to \infty} \lim_{k \to \infty} \rho(f(x_n), f(x'_k)) = \rho(F(\al), F(\beta)) \]
Hence $F$ is an isometric embedding. 
\end{proof}

\begin{example}
This is an example where the extension $F$ of $f$ is not unique. 
Let $X = (0,1]$ with the asymmetric distance from Example \ref{NonSymEx}. Let $Y = X \cup \{0\} \cup\{0'\}$ with $d(t,0) = t = d(t,0')$ for all $t \in X$, and $d(0,t) = d(0',t) = 1$ and $d(0,0') = 0$ and $d(0', 0) =1$. Then $0$ is the CFL of $\{ \frac{1}{n}\}$ and $0'$ is a forward limit of $\{\frac{1}{n}\}$. The inclusion $f \from X \to Y$ is an isometric embedding. $X$ is forward continuous and minimizing, $Y$ is semi forwards continuous with respect to $X$. However, $f$ has two extensions to $\hat X$. One can define $F([\{\frac{1}{n}\}]) = 0$ or to $0'$. 
\end{example}

\begin{prop}\label{uniqueness}
Under the hypothesis of Theorem \ref{extIsom}, if $\text{Im}(f)$ is forwards dense in $Y$ then $F$ is unique. 
\end{prop}
\begin{proof}
Let $F \from \hat X \to Y$ be an isometric embedding that extends $f$. Let $\al \in \hat X$, and $\xi = \{x_n\} \subset X$ a forwards admissible sequence such that $\xi \in \al$. Let $y$ be the CFL of $f(\xi)$. Then since $F$ is an isometric embedding, we have,
\[ \begin{array}{l}
\displaystyle \lim_{n \to \infty} \rho(f(x_n), F(\al)) =
\lim_{n \to \infty} \rho(F \circ \iota(x_n), F(\al)) = \\
\displaystyle \lim_{n \to \infty} \hat d(\iota(x_n), \al) = 
\lim_{n \to \infty} \lim_{k \to \infty} d(x_n, x_k) = 0.
\end{array}\] Because $y$ is the CFL of $\{f(x_n)\}_{n=1}^\infty$ by the definition of CFL we have $\rho(y,F(\al))=0$. 

Conversely, since $\text{Im}(f)$ is forwards dense in $Y$, $F(\al)$ is a CFL of some forwards admissible sequence in $\text{Im}(f)$. Thus there exists a forwards admissible $\zeta = \{z_n\} \subset X$ so that $F(\al)$ is the CFL of $f(\zeta)$. Since $Y$ is semi forwards continuous with respect to $\text{Im}(f)$ and $F(\al) = \fl{n} f(z_n)$ and $y = \fl{n} f(x_n)$ then we have
\begin{equation}\label{eq444} 
\rho(F(\al),y) = \lim_{n \to \infty} \lim_{k \to \infty} \rho( f(z_n), f(x_k)) = 
\lim_{n \to \infty} \lim_{k \to \infty} d( z_n, x_k) 
\end{equation}
Since $F(\al)$ is the CFL of $f(\xi)$ then 
$\lim_{n \to \infty} \rho(f(z_n), F(\al)) = 0$. Therefore, $\lim_{n \to \infty} \hat d(\iota(z_n), \al) = 0$. 
Since $\hat X$ is forwards continuous with respect to $X$, we have 
$\lim_{n \to \infty} \lim_{k \to \infty} d(z_n, x_k) =0$. Along with equation \ref{eq444} we get $\rho(F(\al),y) = 0$. 
\end{proof}

\begin{cor}\label{isomX}
For any forwards continuous asymmetric space $(X,d)$ which satisfies property (\ref{star}), 
each isometry $f \from X \to X$ induces a unique isometry $F \from \hat X \to \hat X$. 
\end{cor}
\begin{proof}
Since $X$ satisfies property (\ref{star}) and is forwards continuous we have that $X$ is minimizing and $\hat X$ is semi-forwards continuous with respect to $X$ (Lemma \ref{keyLemma}). By Lemma \ref{embedding}, $\iota(X)$ is forwards dense in $\hat X$. An isometry $f \from X \to X$ induces an isometric embedding $\iota \circ f \from X \to \hat X$ and by Theorem \ref{extIsom}, $\iota \circ f$ extends to an isometric embedding $F \from \hat X \to \hat X$. Since $\iota(X)$ is forwards dense, $F$ is an isometry. Moreover, by Proposition \ref{uniqueness}, $F$ is unique. 
\end{proof}

We finally prove our main theorem.

\begin{proof}[Proof of Theorem \ref{ThmCompl}]
The existence part of the theorem follows from Lemma \ref{keyLemma}, Definition \ref{defnCompl}, Lemma \ref{bigLemma} and Lemma \ref{embedding}.
Now suppose that $(Y,\rho)$ is a forwards complete asymmetric metric space and $f \from X \to Y$ is an isometric embedding so that $Y$ is semi forwards continuous with respect to $\text{Im}(f)$ and $\text{Im}(f)$ is forwards dense in $Y$. Then by Theorem \ref{extIsom} $f$ extends to an isometric embedding $F \from \hat X \to Y$ and by density $F$ is an isometry. Thus $(Y,\rho)$ is isometric to $(\hat X, \hat d)$. 
\end{proof}

\section{Background on Outer Space} \label{secOuterSpace}

In this section we review the graph and tree approaches to Outer Space. We then formulate a lifting proposition that allows us to lift an affine difference in marking of graphs to an affine equivariant map of $F_n$-trees with specified data, see Proposition \ref{prop_lift}. We end the section with a review of the boundary of Outer Space. 

\subsection{Outer Space in terms of marked graphs}

Fix a basis $\{ x_1, \dots x_n \}$ of $F_n$. When we talk of reduced words we will mean relative to this basis.

\begin{defn}[The rose]
Let $R$ be the wedge of $n$ circles, denote the vertex of $R$ by $*$. We identify $x_i$ with the positively oriented edges of $R$. This gives us an identification of $\pi_1(*,R)$ with $F_n$ that we will suppress from now on. 
\end{defn}
\begin{defn}[Outer Space]
A point in outer space is an equivalence class of a triple $x = (G,\tau,\ell)$ where $G$ is a graph (a finite 1 dimensional cell complex), $\tau \co R \to G$ and $\ell \co E(G) \to (0,1)$ are maps, and $(G,\tau,\ell)$ satisfy:
\begin{enumerate}
\item the valence of $v \in V(G)$ is greater than $2$.
\item $\tau$ is a homotopy equivalence. 
\item $\sum_{e \in E(G)} \ell(e) = 1$.
\end{enumerate}
The equivalence relation is given by: $(G,\tau,\ell) \sim (G',\tau',\ell')$ if there is an isometry $f \co (G,\ell) \to (G',\ell')$ such that $f \circ \tau$ is freely homotopic to $\tau'$. We shall sometimes abuse notation and use the triple notation $(X,\tau,\ell)$ to mean its equivalence class. 
\end{defn}

We will always identify words in $F_n$ with edge paths in $R$, note that reduced words are identified with immersed paths in $R$. Using this identification, an automorphism $\phi:F_n \to F_n$ can be viewed as an affine map $\phi: R \to R$.
\begin{defn}[$\out$ action] 
There is a right $\text{Aut}(F_n)$ action on the set of metric marked graphs given by: $x \cdot \phi = (G, \tau \circ \phi, \ell)$. 
This action is constant on equivalence classes, and inner automorphisms act trivially. Therefore, this action descends to an $\out$ action on $\os$. 
\end{defn}

\subsection{Outer Space in terms of tree actions}
An equivalent description of Outer Space is given in terms of minimal free simplicial metric $F_n$-trees. 
\begin{defn}[Tree definition of $\os$]
Outer Space $\os$ is the set of equivalence classes of pairs $(X,\rho)$ where $X$ is a metric tree, and $\rho: F_n \to \isom(X)$ is a homomorphism and the following conditions are satisfied: 
\begin{enumerate}
\item The action is free - if $\rho(g)(p) = p$ for $p \in X$ and $g \in F_n$ then $g = 1$.

\item $X$ is simplicial - for any $1 \neq g \in F_n$, the translation length 
\[ l(\rho(g),T) := \inf \{ d(x, \rho(g) x) \mid x \in T\}\] 
is bounded away from zero by a global constant independent of $g$.
\item The action is minimal - no subtree of $X$ is invariant under the group $\rho(F_n)$.
\item The action is normalized to have unit volume - $X/\rho(F_n)$ is a finite graph whose sum of edges is $1$.
\end{enumerate}
The equivalence relation on the collection of $F_n$ tree actions is defined as follows: $(X,\rho) \sim (Y, \mu)$ if there is 
an isometry $f \co X \to Y$ with $f^{-1} \circ \mu(g)\circ f (x) = \rho(g)(x)$. In this case 
$(X,\rho)$, $(Y,\mu)$ are called isometrically conjugate.
\end{defn}

\begin{remark}\label{graphQuotient}
The first three items imply that $X/\rho(F_n)$ is a finite metric graph. Indeed, by (1) and (2) the 
action is properly discontinuous therefore $p \co X \to X/\rho(F_n)$ is a covering map.
Since the action is free then $\pi_1(X/\rho(F_n),*) \cong F_n$ and because of (3), there are no valence 1 in $X/\rho(F_n)$. 
\end{remark}

There is an action of $\text{Aut}(F_n)$ on simplicial minimal metric $F_n$-trees 
given by 
\[(X, \rho) \cdot \phi = (X, \rho \circ \phi)\] 
Clearly, the action is 
constant on the equivalence classes. 
To see that inner automorphisms act trivially, assume $\phi = 
i_g$ and take $f = \rho(g) \co X \to X$ then $f$ is an isometry such that 
\begin{align*}
f^{-1} \circ \rho(h) \circ f(x) &= \rho(g)^{-1} \circ \rho(h)(\rho(g)(x)) \\
& = \rho(ghg^{-1})(x) \\
& = [\rho \circ i_g](h) (x)
\end{align*} 
Therefore $\rho$ is isometrically conjugate to $\rho \circ i_g$ and hence the 
action descends to an action of $\out$ on the isometry classes of trees. 
\subsection{An Equivalence of the Categories}
\textbf{From graphs to trees.} The operation of lifting gives us a way of converting a marked graph to an $F_n$-tree. Let 
$(G, \tau, \ell)$ be a metric marked graph, by choosing a point $w$ in the fiber of $\tau(*)$ we obtain 
an action of $\pi_1(G,\tau(*))$ by deck transformations on the universal cover of $G$, $\wt G$. 
By precomposing with $\tau_*$ we obtain a homomorphism $\rho^G_w:F_n \to \isom(\wt G)$. A 
different choice of a point $z \in p^{-1}(\tau(*))$, where $p\from \wt G \to G$ is the covering map, results in $\rho^G_z$ that is isometrically conjugate to $\rho^G_w$. 
A marking $\tau'$ homotopic to $\tau$ 
would produce a homomorphism $\rho'$ isometrically conjugate to $\rho$.

\textbf{From trees to graphs.} There is also an inverse operation, namely, 
taking the quotient of a simplicial metric tree $T$ by a free and simplicial $F_n$ action, see Remark \ref{graphQuotient}. The quotient is a finite metric graph with a marking $\tau$ determined by a basepoint $q \in T$. A change of basepoint would lead to a marking $\tau'$ homotopic to $\tau$. Moreover, the tree is the universal cover of the quotient. Therefore, the operations of lifting and taking quotients are inverses of each other.

Let $x=(G,\tau, \ell), y=(H,\mu,\ell')$ be two points in $\os$ a \emph{difference in markings} is a 
map $f:G \to H$ such that $ f \circ \tau $ is homotopic to $\mu$. Thus, $y$ is equal to the equivalence class of $(H, f \circ \tau, \ell')$ and this will be our preferred representative of the equivalence class. 
\begin{prop}\label{liftingAmap}
Let $f \from G \to H$ be a graph map and 
let $( \wt G,p),(\wt H,p')$ be the universal covers of $G$ and $H$. 
Given a choice of 
basepoints $w \in p^{-1}(\tau(*))$ 
and $z \in p'^{-1}(f \circ \tau(*))$ 
there is a unique lift of $f \circ p:\wt G \to H$ to $\wt{f_{wz}}: \wt G \to \wt H$. 
\begin{enumerate}
\item If $f$ is a difference in marking then for all $h \in F_n$,
\[ \wt{f_{wz}} \circ \rho^G_w(h) = \rho^H_z(h) \circ \wt{f_{wz}} \]
\item If $f$ is affine then $\wt{f_{wz}}$ is affine. 
\end{enumerate}
\end{prop}

We note that the other direction is also true. Given a linear equivariant map ${\sf f}:\wt G \to \wt H$, it descends to a map $f:G \to H$ which is a difference of the induced markings. 

Therefore the lifting operation defines an equivalence of categories between (equivalence classes of) marked metric graphs and linear differences of markings, and the category of (equivalence classes of) $F_n$-trees and linear equivariant maps.

\subsection{Lifting optimal maps}
Let $x$ be a point in Outer Space and let $\al$ be a loop in $x$ (i.e. the underlying graph of $x$), we denote by $l(\al,x)$ the length of the immersed loop homotopic to $\al$. 
For $a \in F_n$ we denote by $l(a,x)=l(\tau(a),x)$ or equivalently the length of each element in the conjugacy class of $a$ in $x$. 
Denote by $X$ be the $F_n$-tree obtained as the universal cover of $x$. Since the action is free each $a \in F_n$ acts as a hyperbolic isometry and its translation length is equal to $l(a,x)$. It has an axis denoted by $A_X(a)$. 

\begin{defn}
A loop $\al$ in $x$ is a \emph{candidate} if it is an embedded circle, an embedded figure 8, or an embedded barbell. \end{defn}

\begin{theoremDT}\cite{FM}
Let $x,y \in \os$ then the function
\begin{equation} \label{eq23}
\begin{array}{ll}
d(x,y) & = \log \inf \{ \Lip(f) \mid f \co x \to y \text{ is a Lipchitz difference in markings }\} \\ 
& = \log \sup \left. \left\{ \frac{l(\gamma,y)}{l(\gamma,x)} \right| \gamma \text{ is a loop in } x \right\}
\end{array} \end{equation}
defines an asymmetric distance on $\os$. Additionally, the supremum and infimum in equation \ref{eq23} are realized. 
\end{theoremDT}

\begin{defn}
A loop which realizes the maximum in equation \ref{eq23} is called a \emph{witness}. Sometimes we shall call a witness the element of $F_n$ or its conjugacy class that corresponds to the witness loop. A map that realizes the minimum in equation \ref{eq23} is called an \emph{optimal map}.
\end{defn}

\begin{prop}\cite{FM}\label{witnessOptimal}
For $x,y \in \os$ there is a candidate $\al$ of $x$ that witnesses the distance $d(x,y)$. 
Moreover, if $\beta$ is a witness and $f \from x \to y$ is an optimal difference in marking then $f(\al)$ is an immersed loop in $y$ and $d(x,y) = \log \frac{l(f(\al),y)}{l(\al,x)}$. 
\end{prop}

One can compute the distance 
$d(x,y)$ by going over all candidates in $x$ and finding those which stretch maximally. 

\begin{defn}
Let $x \in \os$ a basis $\mathcal{B}$ of $F_n$ is \emph{short} with respect to $x$ if for every $a \in \mathcal{B}$, $l(\tau(a),x)\leq 2$. 
\end{defn}

The following proposition follows from standard covering space theory.

\begin{prop}\label{prop_lift}
Let $x = (G, \tau, \ell)$ and $y=(H, \mu, \ell')$ be elements of $\os$, and 
$f:x \to y$ an affine optimal map. Let $\beta$ be a candidate witness for the distance $d(x,y)$, and let $b = \tau_*^{-1}(\beta) \in F_n$. Then for any choice of basepoint $w'$ in the image of the loop $\beta$ in $G$ and for any choice of lift $w \in \wt G$ of $w'$ and for any choice of lift $z \in \wt H$ of $f(w')$ on the axis $A_{\wt H}(b)$ the lift $\wt{f_{wz}}:\wt G \to \wt H$, from Proposition \ref{liftingAmap}, restricts to a linear map from 
$A_{\wt G}(b)$ to $A_{\wt H}(b)$. \end{prop}


\subsection{The boundary of Outer Space}
The advantage of the tree approach to Outer Space is that 
it extends to a compactification of Outer Space. Given an $F_n$ tree $X$ consider the function of translation lengths in $X$
\begin{align*} 
\ell_X : & F_n \to \RR \\
& \ell_X(a) = \ell(\rho(a),X) 
\end{align*}
Culler and Morgan \cite{CM} defined a set of five axioms and called a function 
$\ell: F_n \to \RR$ a \emph{pseudo length functions} if it satisfied their axioms. Any 
length function of an $F_n$ tree is a pseudo length function. They proved 
\begin{theorem}\cite{CM}\label{uniqueTree}
If a length function $\ell$ is irreducible (there are $g,h \in F_n$ with $\ell(h), \ell(g)$ and $\ell([g,h])$ non-zero) then there exists a \textbf{unique} minimal tree $X$ so that $\ell = \ell_X$ and this tree is irreducible (i.e. there is no global fixed point, no invariant end, and no invariant line). 
\end{theorem}
Let $\plf$ the space of projective length functions be the quotient of the space of length functions, considered as a subspace of $\RR^{F_n}$, quotioned by the action of $\RR$. By the previous paragraph, there exists an injection 
\begin{equation}\label{eqInjection} \os \to \plf \end{equation}

\begin{defn}\label{axesTopo}
The topology of $\os$ induced by the injection \ref{eqInjection} is called the \emph{axes topology}.
\end{defn}

Culler and Morgan \cite{CM} proved 
that $\plf$ is compact. The closure $\overline{\os}$ of $\os$ in $\plf$ is called the compactification of 
Outer Space. Cohen and Lustig \cite{CohL} and Bestvina and Feighn \cite{BF} showed that the compactification of 
Outer Space $\overline \os$ can be characterized as the set of equivalence classes of 
very small $F_n$-trees.

\begin{defn}\label{defnVerySmall}A very small $F_n$-tree is a metric $\RR$-tree $T$ and a homomorphism 
$\rho: F_n \to \text{Isom}(T)$ so that 
\begin{enumerate}
\item $\rho$ is minimal and irreducible.
\item For every set ${\sf s} \subset T$ that is isometric to an interval, $\stab{\sf s}$ is cyclic or trivial. If $\stab{{\sf s}} = \langle g \rangle$ then $g$ is not a power.
\item For every tripod $\sf t$ in $T$, $\stab{\sf t} = \{1\}$.
\end{enumerate}
Two very small $F_n$ trees are equivalent if there is an equivariant isometry between them.
\end{defn}

Free simplicial $F_n$-trees are very small trees and those are precisely the points of 
Outer Space. 

There is another topology on $\overline{\os}$ called the \emph{Gromov topology}.

\begin{defn}
A basis element $O(T,P,K,\eps)$ is parameterized by an $\RR$-tree 
$T$, a compact subset $K$ in $T$, a finite subset $P< F_n$ and $\eps>0$. 
A \emph{$P$-equivariant $\eps$-relation} $R$ between $K \subset T$ and $K' \subset T'$ is a subset 
$R \subset K \times K'$ so that the following hold
\begin{enumerate}
\item the projection of $R$ is surjective on each factor, 
\item if $(x,x') , (y,y') \in R$ then $|d(x,y) - d(x',y')|< \eps$,
\item for all $a \in P$ if $x$ and $\rho(a) x \in K$ 
and $(x,x') \in R$ then $\rho'(a)x' \in K'$ and $(\rho(a)x,\rho'(a)x') \in R$.
\end{enumerate}
A basis element $O(T,K,P,\eps)$ for the Gromov topology is the 
set of trees $S$ so that there is a compact $K' \subset S$ and a $P$-equivariant 
$\eps$-relation $R \subset K \times K'$. 
\end{defn}
Paulin \cite{P} showed that the axes topology and the Gromov topology coincide on $\os$. They are both equal to the simplicial topology defined by using the $l_1$ metric on the following simplicies.

\begin{defn}\label{simplTopo}
Let $G$ be a graph with no vertices of valence $<3$ and $\tau:R_n \to G$ a marking. The simplex $S_{G,\tau} \subset \os$ is the set of points $(G, \tau,\ell)$ with $\ell$ ranging over all possible functions 
$\ell:E(G) \to (0,1)$ with $\sum_{e \in G} \ell(e) = 1$. 
The $l_1$ metric between $(G,\tau,\ell), (G, \tau, \ell')$ in this simplex is $\sum_{e \in G} |\ell(e) - \ell'(e)|$. A face of $S_{G,\tau}$ corresponds to a forest collapse. 
This defines a simplicial structure on $\os$ that is locally finite. The topology induced on $\os$ is called 
\emph{the simplicial topology}. \end{defn}

\section{A characterization of the completion points}

\subsection{Outer Space admits a forwards completion.}

We will first prove that Outer Space is backwards complete but not forwards complete. 
Next we will show that it satisfies the conditions of Theorem \ref{ThmCompl} and thus it admits a forwards completion.
Finally we will characterize the points of the compactification of Outer Space that are completion points.

\begin{defn}\label{defBInOut}
Let $x \in \os$, $0 \leq r \in \RR$ the incoming ball centered at $x$ with radius $r$ is
\[ B_\text{in}(x,r) = \{y \in \os \mid d(y,x) < r\}. \] 
The closed ball is $\clBin(x,r) = \{ y \in \os \mid d(y,x)\leq r\}$. Similarly, the outgoing ball centered at $x$ with radius $r$ is
\[ B_\text{out}(x,r) = \{y \in \os \mid d(x,y) < r\} \]
and $\overline{B}_\text{out}(x,r) = \{ y \in \os \mid d(x,y)\leq r\}$ is the closed ball.
\end{defn}

\begin{prop}\label{compactBalls}
$\clBin$ is compact (while $\overline{B}_\text{out}(x,r)$ is not).
\end{prop}
\begin{proof}
Consider $y \in \Bin$ since $d(y,x)< r$ then for all $a \in F_n$, $l(a,y) > \frac{1}{r} l(a,x)$. Let $s$ be the length of the shortest loop in $x$, then $l(a,y) > \frac{s}{r}$. This implies that $y$ lies in the $\frac{s}{r}$ thick part of $\os$ and by \cite{AKB}, there exists a constant $A = A(s,r)$ so that $d(x,y) \leq A d(y,x) <Ar$. Thus, for all $a \in F_n$ we also have $l(a,y) < Ar l(a,x)$. In conclusion,
\[ \frac{1}{r} l(a,x) < l(a,y) < Ar l(a,x) \]
Now suppose $z$ lies in the closure of $\Bin$ in the space of projective length functions. Let $\{x_k \}$ be a sequence in $\Bin$ projectively converging to $z$, thus 
\[ \frac{l(a,z)}{l(b,z)} = \lim_{k \to \infty} \frac{l(a,x_k)}{l(b,x_k)} \text{ for all } a,b \in F_n \]
Fix $b \in F_n$ then $l(b,x)>0$. Thus, for all $a \in F_n$, 
\[ \frac{1}{Ar^2} \frac{l(b,z)}{l(b,x)} l(a,x) \leq l(a,z) \leq Ar^2 \frac{l(b,z)}{l(b,x)} l(a,x). \]
In particular $z$ corresponds to a free and simplicial $F_n$ tree and so $z$ lies in Outer Space. Thus, the closure of $\Bin$ in the space of projective length functions is equal to its closure in Outer Space. Since the former is compact we obtain our statement. 

Finally, to see that $\clBout$ is not compact, let $\al$ be an embedded loop in $x$ and let $x_k$ be a graph so that the edges of $\al$ are rescaled by $\frac{1}{k}$ and the other edges are rescaled so the total volume is 1. It is elementary to see that $d(x,x_k)$ is at most $r = \frac{1}{s}$ where $s$ is the length of the shortest edge in $x$. Hence $x_k \in B_\text{out}(x, r)$ but $x_k$ converges projectively to a non free $F_n$ tree hence the limit lies outside Outer Space. 
\end{proof}

\begin{lemma}\label{addmissibleConv}
If $\{x_k\} \subset \os$ is a forwards admissible sequence then for each $a \in F_n$, $l(a,x_k)$ is almost monotonically decreasing. 
If $\{x_k\} \subset \os$ is a backwards admissible sequence then for each $a \in F_n$, $l(a,x_k)$ is almost monotonically increasing. 
\end{lemma}
\begin{proof}
We know that for all $\eps$ there exists an $N(\eps) \in \NN$ so that for all $m>N(\eps)$ there is a $K(m,\eps) \in \NN$ such that $d(x_m,x_k) < \log(1+\eps)$. Therefore, for all $a$, $\frac{l(a,x_k)}{l(a,x_m)}< e^{\log(1+\eps)} = 1+ \eps$. Hence $l(a,x_k) < l(a,x_m) + \eps l(a,x_m)$. Now let $\eps=1$ then for $S = N(1)+1$ denote $M = l(a,x_S)$ then for all $m>K(S,1)$ we have $l(a,x_m) < 2M$. Now let $\eps>0$ be arbitrary and take $m> \max\{N(\frac{\eps}{M}), K(S,1)\}$ then there exists a $K(m,\frac{\eps}{M})$ such that for all $k>K$ we have
\[ l(a,x_k)< l(a,x_m) + \frac{\eps}{M} l(a,x_m) < l(a,x_m) + \eps. \]
If $\{x_k\}$ is backwards admissible then for all $\eps$ there exists an $N(\eps) \in \NN$ such that for all $m>N(\eps)$ there is a $K(m,\eps) \in \NN$ so that for all $k>K(m,\eps)$ we have $\frac{l(x_m,a)}{l(x_k,a)}<1+\eps$. We proceed as in the case of a forwards admissible sequence.
\end{proof}

\begin{prop}\label{BackAddConv}
If $\{x_k\}$ is a backwards admissible sequence then it has a unique forwards limit $x$. The point $x$ is also the limit of $\{x_k\}$ in the axes topology.
\end{prop}
\begin{proof}
If $\{ x_k\}$ is backwards admissible then $\{x_k\} \subseteq \clBin(x_0, r)$ for a large enough $r$. Thus, $\{x_k\}$ has a partial limit $x \in \os$. Fix a conjugacy class $\al$ and consider the sequence $l(\al,x_k)$, since $x_k$ is backwards admissible, $l(\al,x_k)$ is almost monotonically increasing. Moreover, if $\{x_{k_i}\}$ is the subsequence that limits to $x$ then $l(\al, x_{k_i})$ limits to $l(\al,x)$ and therefore, $l(\al,x_k)$ is bounded. Therefore $l(\al,x_k)$ converges by Proposition \ref{Prop:BoundedAndAlmostMonDec} and to $l(\al,x)$. Thus, the partial limit $x \in \os$ is in fact the limit of $\{x_k\}$ in the axes topology. If $y \in \os$ so that $\lim_{m \to \infty} d(x_k,y) = 0$ then $d(x,y) = 0$ (since the length functions converge) and thus, $y=x$. 
\end{proof}

\begin{prop}
$\os$ is backwards complete but not forwards complete.
\end{prop}
\begin{proof}
Outer Space is backwards complete by Proposition \ref{BackAddConv}. 
To see that it is not forwards complete consider the sequence from the proof of Proposition \ref{compactBalls} showing that outgoing balls are not compact and notice that it is a forwards Cauchy sequence. 
\end{proof}

To prove that $\os$ has a forwards completion guaranteed by Theorem \ref{ThmCompl} we must show that it is forwards continuous (Definition \ref{contDefn}), minimizing (Definition \ref{miniDefn}) and that $\hat X$ is semi forwards continuous with respect to $X$ (Definition \ref{semiForwardsContinuous}). By Lemma \ref{keyLemma} it suffices to show that $X$ is forwards continuous and that it satisfies condition \ref{star} of Lemma \ref{keyLemma} (we recall this condition at the end of the proof of Proposition \ref{conditions}). 

\begin{prop}\label{conditions}
Outer Space is forwards and backwards continuous. Moreover, $X$ satisfies condition \ref{star} of Lemma \ref{keyLemma} and therefore it is minimizing and $\hat X$ is semi forwards continuous with respect to $X$.
\end{prop}
\begin{proof}
Suppose that $\{x_k\}$ is a sequence in Outer Space with a forward limit $x \in \os$. Then there exists $K\in \NN$ such that for all $k>K$ we have $d(x_k,x)<1$. Thus for each conjugacy class $\al$ in $F_n$ we have $l(a,x_k)> e^{-1} l(a,x)$. 
Since $x \in \os$ its length function is positive and bounded away from zero, so there exists $\eps_0$ so that $x_k \in \os^{\geq \eps_0}$ for all $k>K$. Therefore there exists $A>0$ such that $d(x,x_k) \leq A d(x_k,x)$ for all $k>K$. Thus, $x$ is also a backwards limit of $\{x_k\}$. 
Therefore, $x$ is the limit of $\{x_k\}$ (by Proposition \ref{BackAddConv}). 

Now for every $y \in \os$, we have $d(x_m,y) - d(x,x_m) \leq d(x,y) \leq d(x,x_m) + d(x_m,y)$ and hence $\displaystyle \lim_{m \to \infty}d(x_m,y) \leq d(x,y) \leq \lim_{m \to \infty} d(x_m,y)$. Similarly $d(y,x) = \lim_{m \to \infty} d(y,x_m)$. Hence $\os$ is forwards continuous. The proof that it is backwards continuous is analogous and left to the reader. 

Condition \ref{star} states that every forwards admissible sequence $\{x_m\}$ in $\os$ that has a CFL in $\os$ denoted $x$ is also backwards admissible and $x$ is its CBL. This is precisely what we have already shown. 
\end{proof}

\begin{cor}
Outer Space admits a forwards completion.
\end{cor}

In the next sections we will work to characterize this completion as a subset of $\overline{\os}$.

\subsection{The limit of a forward admissible sequence}
Recall that $\ol{\os}$ is compact. In general, if $\{X_m\}$ is a sequence in $\ol{\os}$ then there exists a subsequence $\{X_{m_j}\}$ and a sequence of scalars $\{c_j\}$ so that $\lim_{j \to \infty} c_j X_{m_j}$ converges to $X \in \ol{\os}$ (in the space of length functions). 
We observe that a forward Cauchy sequence in $\os$ converges, and without the need to rescale the trees, i.e. we can take $c_j =1$ for each $j$.

\begin{cor}\label{limitExists}
Let $\{X_k\} \subset \os$ be a forwards Cauchy sequence then there exists an $F_n$-tree $X$ whose homothety class lies in Outer Space, so that $\lim_{k \to \infty} X_k = X$ as length functions.
\end{cor}
\begin{proof}
Let $\| \cdot \|_k$ be the translation length function of $X_k$. 
For each conjugacy class $\al$, the sequence $\| \al \|_{k}$ is positive and almost monotonically decreasing (Lemma \ref{addmissibleConv}). Therefore $\| \cdot \|_{k}$ converges to a length function, that is a length function of a very small tree action \cite{CohL}.
\end{proof}

We denote by $x_m$ the $F_n$-quotient of $X_m$ and let $h_{m,k}\from x_m \to x_k$ be an optimal difference in marking. Then $\log \Lip h_{m,k} = d(x_m,x_k)$. Fix $m$ and let $k$ vary, then $\Lip h_{m,k}$ is an almost monotonically decreasing Cauchy sequence. 

\begin{cor}\label{limitLipConst}
If $\{x_m\}$ is an admissible sequence and $h_{m,k} \from x_m \to x_k$ optimal differences in marking then $\{ \Lip h_{m,k} \}_{k = m}^\infty$ converges to some limit $L_m$ and it is bounded by $M_m$.
\end{cor}

Our next goal is to show that for each $m$ there exists a map $\f{m,\infty}: X_m \to X$ such that $\Lip \f{m,\infty} \leq L_m$. 
The limit $X$ is given in terms of length functions so we need to work in order to give another description of it that would easily admit a map. 

For the remainder of this section, we fix $1 \leq m \in \NN$, following closely Bestvina's construction\footnote{In fact the rest of this section is a repetition of the arguments of \cite{degen} in our setting (that is a little different than the original).} in \cite{degen}, and using Gromov's Theorem \ref{GromovThm}, we construct a metric tree $X_\infty(m)$, with an $F_n$-action, and an equivariant map $f_{m, \infty} \from X_m \to X_\infty(m)$.
All of these constructions depend on $m$. However, we will show that the tree $X_\infty(m)$ is non-trivial, irreducible and that the length functions of $X_k$ converge to that of $X_\infty(m)$. Therefore, by Theorem \ref{uniqueTree}, $X_\infty(m)$ is equivariantly isometric to the limit $X$ from Corollary \ref{limitExists}.

We construct $X_\infty(m)$ as a union of finite trees $X^l_\infty(m)$ as follows. For every $k>m$ there is a candidate $\beta_k$ in $x_m$ that is stretched maximally by the difference in marking $h_{m,k} \from x_m \to x_k$. 
By passing to a subsequence we may assume that it is 
the same for all $k>m$ and denote it by $\beta$. 
Note that $\beta$ depends on $m$ but since we have now fixed $m$, we shall drop the $m$-index to simplify the notation. 
We denote by $\tau_k \from R \to x_k$ the marking of $x_k$. Let $\B$ (also depends on $m$) be a basis such that for every $c \in \B$, $\tau_m(c)$ is homotopic to a candidate of $x_m$ and there is a $b$ in $\B$ such that $\tau_m(b) = \beta$. 
Recall that $h_{m,k}(\beta)$ is immersed in $x_k$. Let $A_m(b)$ be the axis of $b$ in $X_m$. Let $w \in X_m$ be a point on $A_m(b)$ ($w$ depends on $m$), let $w'$ be the projection in $x_m$ of $w$, and choose 
$w_k \in X_k$ be a point on $A_k(b)$ in the preimage of $h_{m,k}(w')$ (here $w_k$ depends on $m$ as well, but again we repress this dependence).
By Proposition \ref{prop_lift} there exists a map $f_{m,k} \from X_m \to X_k$ that lifts $h_{m,k}$ and so that $f_{m,k}(w) = w_k$ and $f_{m,k}|_{A_m(b)}$ is a linear map onto $A_k(b)$.



\begin{prop}
For each basis element $c \in \B$, the $f_{m,k}$ image of the axis $A_m(c)$ is contained in the $\Lip(f_{m,k})$ neighborhood of 
the axis $A_k(c)$.
\end{prop} 
\begin{proof}
The cycle $\gamma = \tau_m(c)$ is a candidate in $x_m$. Therefore, it is short, i.e. $l(c,x_m) \leq 2$ and hence $l(c,x_k) \leq 2 \Lip(f_{m,k})$. Part of $f_{m,k}(\gamma)$ survives after tightening the loop. This is the part that lifts to a fundamental domain of $A_k(c)$. Therefore, $f_{m,k}(A_m(c))$ is contained in the $2 \Lip(f_{m,k})$ neighborhood of $A_k(c)$.
\end{proof}

We exhaust $F_n$ via $\BB$-length: 
\[ W^l = \{ g \in F_n \mid |g|_{\BB^m}\leq l\}\] 
Denote by $X_k^l$ the convex 
hull of $\{ W^l \cdot z_k \}$ . A \emph{diagonal} in $X_k^l$ is a path of the 
form $[\rho(g) w_k, \rho(h) w_k]$. Each diagonal in $X_k^l$ can be covered 
by $\frac{2M_ml}{\eps}$ balls of radius $\eps$ because:
\[d_{X_k}(w_k, \rho_k(g) w_k) \leq \Lip(\fmk) d(w_m ,\rho_m(g) w_m) \leq \Lip(\fmk) l \leq M_m l \] 
Note that this number is uniform over all $k$. We now apply 
\begin{theorem}\cite{Gr}\label{GromovThm}
If $\{Y_k\}_{k=k_0}^\infty$ is a sequence of compact metric spaces so that for every $\eps$ there is an $N(\eps)$ so that $Y_k$ may be covered by $N(\eps)$ $\eps$-balls then there is a subsequence $Y_{k_j}$ which converges in the Gromov sense to a compact metric space.
\end{theorem}
Therefore, for each $l \in \NN$ we denote the limit space provided by the theorem $X^l_\infty(m)$ and the sequence $\{k_j^l\}_{j=1}^\infty$ so that $\lim_{j \to \infty} X^l_{k^l_j} = X^l_\infty(m)$. By a diagonal argument we may pass to a subsequence $k_j$ so that $X^l_{k_j}(m)$ 
converges to $X^l_\infty(m)$ for every $l$. 

\begin{defn}\label{actionDefn}
We require of the sequence $\{k_j\}$ that for each $g \in F_n$ the sequence $\rho_k(g)(w_{k_j}) \in X_{k_j}$ converges in $X^l_\infty(m)$ and denote its limit by $\rho(g) \cdot w$. The existence of the $k_j$ sequence follows by a diagonal argument from the fact that $F_n$ is countable. 
\end{defn}
\begin{prop}
$X^l_\infty(m)$ is a finite tree.
\end{prop}
\begin{proof}
The proof is identical to that of Lemma 3.5 in \cite{degen}. One shows that the limit of 0-hyperbolic spaces is a 0-hyperbolic space. It is elementary to show that 
\[ X^l_\infty(m) = \bigcup_{g,g' \in W^l} [g \cdot w, g' \cdot w]. \]
This implies that $X^l_\infty(m)$ is connected and that it is a finite tree. 
\end{proof}

The definitions of $X_\infty^l(m)$ and $k_j$ imply that $X_\infty^l(m) \subset X_\infty^{l+1}(m)$. Moreover, the inclusion sends the basepoint $w$ in $X_\infty^l(m)$ to the basepoint in $X_\infty^{l+1}(m)$ and the point $\rho(g) \cdot w$ of $X^l_\infty(m)$ to $\rho(g) \cdot w$ of $X^{l+1}_\infty(m)$. 
\begin{defn}
Let 
$\Xm = \cup_{l=1}^{\infty} X_\infty^l(m)$. 
Then $\Xm$ is a tree and there is a well defined basepoint $w \in \Xm$ that is the image of each basepoint $w \in X_\infty^l(m)$, and the points $\rho(g) \cdot w$ for $g \in F_n$ are also well-defined. 
\end{defn}

In Propositions \ref{Fntree}, \ref{Limitminimal} and \ref{XmIsX} to simplify notation we write $X_k$ instead of $X_{k_j}$ but we always mean that we restrict our attention to the subsequence. 

\begin{prop}\label{Fntree}
There exists a homomorphism $\rho: F_n \to \isom(\Xm)$ so that for every $q \in \Xm$ and for every sequence $\{ q_k \} \subset X_k$ so that $\lim_{k \to \infty} q_k = q$ the following equation holds 
\[ \rho(g) q = \lim_{k \to \infty} \rho_k(g) q_k \]
\end{prop}
\begin{proof}
We follow the proof in \cite{degen} Proposition 4.1 and Theorem 4.2. 
We have already defined the action of $F_n$ on $w$ in Definition \ref{actionDefn}. We check that it is an isometry on orbits of $w$, 
\begin{eqnarray*}
d \mathopen{} \left( \rho(g) \Bigl( \rho(h) w \Bigr), \rho(g) \Bigl( \rho(h') w \Bigr) \right) &= &\displaystyle \lim_{k \to \infty} d(\rho_k(g) \rho_k(h) w_k , \rho_k(g) \rho_k(h') w_k)\\
& = & \displaystyle \lim_{k \to \infty} d(\rho_k(h) w_k , \rho_k(h') w_k) \\
& = & \displaystyle \lim_{k \to \infty} d(\rho(h) w_k, \rho(h') w_k) 
\end{eqnarray*}

For a point $q \in \Xm$ there is some $l$ such that $q \in X^l_\infty(m)$. There is a sequence 
$q_k \in X^l_k$ so that $\lim_{k \to \infty}q_k = q$. Let $g \in W^s$, then $q_k \in [\rho_k(g') w_k, \rho_k(g'') w_k]$ for $g',g'' \in W^l$ implies that 
$\rho_k(g)(q_k) \in [\rho_k(gg') w_k, \rho_k(gg'') w_k]$ so $\rho_k(g)(q_k) \in X_k^{s+l}$. We define 
\[ \rho(g) q = \lim_{k \to \infty} \rho_k(g)(q_k)\] in $X_\infty^{l+s}(m)$. We note that $\rho(g)$ behaves well with respect to distances to orbit points of $w$, 
\begin{eqnarray*}
\displaystyle d(\rho(g)q, \rho(h) w) &= &\displaystyle \lim_{k \to \infty} d(\rho_k(g) q_k, \rho_k(h) w_k) \\
&\displaystyle = &\lim_{k \to \infty} d( q_k, \rho_k(g^{-1}h) w_k) \\
& \displaystyle= & d(q,\rho(g^{-1}h) w) \\
& = & \lim_{k \to \infty}d(q_k, \rho_k(g^{-1}h) w_k) 
\end{eqnarray*}
The point $\rho(g)q$ is determined by the distances above
and it is independent of the sequence $\{q_k\}$.
We now have to prove that $\rho(g)$ is an isometry. Let $q \in \Xm$ and 
suppose $q'$ is in the orbit of $w$, i.e. $q' = \rho(h) w$ for some $h \in F_n$. Then $\rho(g)$ preserves the distance, indeed, $d(\rho(g) q, \rho(g) q') = d(\rho(g) q, \rho(g) \rho(h)w) = \lim_{k \to \infty} d(q_k, \rho_k(g^{-1} gh) w_k) = d(q, \rho(h) w)$. 
More generally for $q,q' \in \Xm$ let $w',w''$ be points in the orbit of $w$ such that $q,q'$ lie on the diagonal $[w',w'']$. $\rho(g)$ preserves the distance between $w,w'$ as well as the distances between $q$ and $w,w'$ and $q'$ and $w',w''$. Since $\Xm$ is a tree, it must preserve the distance $d(q,q')$. 
\end{proof}
\begin{prop}\label{Limitminimal}
$\Xm$ is non-trivial and minimal.
\end{prop}
\begin{proof}
Since $w_m \in A_{m}(b)$ and $f_{m,k}(w_m) = w_k$ and $f_{m,k}$ restricts to an affine map on $A_{m}(b)$, then $w_k \in A_{k}(b)$. Therefore, for all $k$, 
\[ d(w_k, \rho_k(b^2)(w_k) = 2 d(w_k, \rho_k(b)w_k)\] 
The derivative of $\fmk$ on $A_m(b)$ is $\Lip(\fmk) \geq 1$, hence $d(w_k, \rho_k(b) w_k)>d(w_m, \rho_k(b) w_m)$ therefore
\begin{equation}\label{AxisOfb}
d(w, \rho(b^2) w) = 2 d(w, \rho(b)w) \geq 2d(w_m,\rho_m(b)w_m)>0 \end{equation}
Hence $\rho(b)$ is a hyperbolic isometry and $w$ lies on the axis of $\rho(b)$ in $\Xm$.
The tree is minimal: If $H$ is an invariant subtree then it must contain the axis of $\rho(b)$ in $\Xm$ and its orbit under $F_n$. Since $H$ is connected it must also contain the convex hull of this set. By construction $\Xm$ is precisely the convex hull of an orbit of $w$ hence $H=\Xm$. 
\end{proof}

\begin{prop}\label{XmIsX} 
For every $g \in F_n$: 
$\|g\|_{\Xm} = \displaystyle \lim_{k \to \infty} \|g\|_k$
\end{prop}

\begin{proof} By definition
\[ \displaystyle \| g \|_{\Xm} = \lim_{s \to \infty} \frac{d(w, \rho(g^s) w) }{s}= \lim_{s \to \infty} \lim_{k \to \infty} \frac{d(w_k, \rho(g^s) w_k)}{s}\] 
Observe that 
$d(w_k, \rho(g^s) w_k)= 2 d(w_k, A_k(g)) + s \|g\|_k$ and 
\[d(w_k, A_k(g)) \leq \Lip(f_{m,k}) d(w_m, A_m(g)) \leq M_m d(w_m, A_m(g)) = D(g,m)\] 
Therefore, $s \|g\|_k \leq d(w_k, \rho(g^s) w_k) \leq 2D(g,m) + s \|g\|_k$. Hence,
\[
\displaystyle \lim_{k \to \infty} \| g\|_k \leq 
\| g\|_{\Xm} \leq \lim_{k \to \infty} \lim_{s \to \infty} \frac{1}{s} (s \|g\|_k + 2D(g,m)) = \lim_{k \to \infty} \| g \|_k \qedhere \] \end{proof}

We now consider the entire sequence $\{X_k\}$ and not the subsequence $\{X_{k_j}\}$. 

\begin{theorem}\label{mapExists}
Let $\{ X_m\}$ be a forward Cauchy sequence in Outer Space, then there is a point $X \in \overline{\os}$ so that $X$ is the limit of $\{X_m\}$ in the axes topology and 
there is an equivariant Lipschitz map $f_{m,\infty}: X_m \to X$ with $\Lip(f_{m, \infty}) = \lim_{k \to \infty} \Lip(\fmk)$.
\end{theorem}
\begin{proof}
The first part of the statement is Corollary \ref{limitExists}. By Proposition \ref{XmIsX}, the 
length functions of $X$ and $\Xm$ coincide. This length function is non-abelian and both $X$ and $\Xm$ are irreducible $F_n$-trees, hence by Theorem \ref{uniqueTree}, there is an equivariant 
isometry $k_m: \Xm \to X$. 
Let $f_{m,\infty}: X_m \to X$ be the equivariant affine map that sends $w_m \to w$ whenever $m = k_j$ of the subsequence and otherwise let $k_j$ be the smallest index closest to $m$ and let $f_{m,\infty}$ be the composition $X_m \to X_{k_j} \to X_\infty$. 
Then $\Lip f_{m,\infty} = \lim_{k \to \infty} \Lip f_{m,k}$. 
The required maps $X_m \to X$ are $k_m \circ f_{m,\infty}$ and are denoted by $f_{m,\infty}$ as well. 
\end{proof}

\subsection{A characterization of completion points}
Once we have a map $f_{m, \infty}: X_m \to X$ it is straightforward to characterize $X \in \Cos$. We begin with a definition of a quotient volume of an $F_n$-tree.

If $V$ is a finite metric tree then $V = \sqcup \sig_i$ a finite union of arcs 
$\sig_i$ with disjoint interiors. The volume of $V$ is the sum of lengths of $\sig_i$. 
It is easy to see that the volume does not depend on the decomposition of $V$ into non-overlapping arcs. 
\begin{defn}\label{tree_volume}
Let $T$ be an (infinite) $F_n$-tree. The quotient volume of $T$ is 
\[ qvol(T) = \inf \{ vol(V) \mid V \subset T \text{ finite forest and } F_n \cdot V = T \} \]
\end{defn}

\begin{prop}\label{howtogen}
If $h: R \to T$ is an $L$-Lipschitz equivariant map then \[ qvol(T) \leq L \cdot qvol(R)\]
\end{prop}
\begin{proof} For each subset $V \subset R$, 
$vol(h(V)) \leq L vol(V)$. 
Moreover, if $F_n \cdot V = R$ then the orbit of $h(V)$ covers $T$. Therefore,
$qvol(T) \leq L qvol(R)$. 
\end{proof}
If $S$ is a simplicial tree then $qvol(S)$ is equal to the sum of lengths of edges 
of $S/F_n$. More generally, 

\begin{prop}
Let $T$ be an $F_n$-tree, $U \subset T$ a finite subtree and \[P=\{ g \in F_n \mid gU \cap U \neq \emptyset \}\] then $P$ generates $F_n$ if and only if $F_n \cdot U = T$.
\end{prop}
\begin{proof}
Suppose $P$ generates $F_n$, we show that $F_n \cdot U$ is connected and by minimality it must coincide with $T$. Let $W^k$ be the set of words $g \in F_n$ that can be written as $g=p_1 \dots p_k$ for $p_i \in P$. For each $i$, $p_1 \cdots p_{i-1} U \cap p_1 \cdots p_i U \neq \emptyset$. Therefore, $\cup_{k} \cup_{g\in W^k} gU = F_n \cdot U$ is connected. \\
Conversely, suppose $F_n \cdot U = T$, choose a basepoint $u \in U$ and consider $gu$. Let $g_1, \dots g_k \in F_k$ so that the geodesic from $u$ to $gu$ passes linearly through $g_1U, \dots , g_kU$ and $g_k = g$. Since $g_iU \cap g_{i+1}U \neq \emptyset$ we have $g_i^{-1} g_{i+1} \in P$. Moreover, since $u \in U$ then $g_1 \in P$, thus 
\[ g = g_1 \cdot g_1^{-1} g_2 \cdot g_2^{-1} g_3 \dots g_{k-1}^{-1} g_k \] is a multiple of elements in $P$.
\end{proof}

\begin{prop}\label{FindUandS}
For every $F_n$-tree $T$ such that $[T] \in \ol\os$ and for every $\eps>0$ there is a finite and connected subtree $U$, and a finite generating set $\calS$ of $F_n$ such that 
\begin{enumerate}
\item For each $g \in \calS$, $gU \cap U \neq \emptyset$, and
\item $vol(U) \leq qvol(T) + \eps$.
\end{enumerate}
Additionally, there is a simplicial $F_n$-tree $T'$ admitting an equivariant 1-Lipchitz quotient map $p: T \to T'$, so that $qvol(T') = qvol(T)$.
\end{prop}
\begin{proof}
We will use a result of Levitt \cite{Lev} about countable groups acting on $\RR$-trees such that the action is a $J$-action. This condition is satisfied by every $F_n$ tree $T$ whose homothety class $[T]$ lies in the closure of Outer Space. Levitt shows in the Lemma on page 31 that if $B \subset T$ is the set of branch points of $T$, then the number of $F_n$ orbits of $\pi_0(T-\ol{B})$ is finite. Each connected component of $T-\ol{B}$ is an arc. We choose arcs $\sig_1, \dots, \sig_k$ so that the union of their $F_n$ orbits cover $T-\ol{B}$. Let $\theta_1, \dots , \theta_k$ be the associated lengths, and let $\eps>0$ be smaller than $\frac{\theta_i}{2}$ for each $i$. 
Consider $N_1 \subset T$ the $\eps$-neighborhood of $F_n \cdot \sig_1$. If $N_1$ is connected then by minimality $N_1 = T$. But the midpoint of $\sig_2$ cannot be contained in $N_1$ since $\eps< \frac{\theta_2}{2}$, hence $k=1$. 
Otherwise, $N_1$ is not connected. Denote by $S_1$ the $\eps$-neighborhood of $\sig_1$. Therefore, there exists an $i$ and an element $a \in F_n$ so that $S_1 \cap a \cdot \sig_i \neq \emptyset$. We reorder the arcs so that $i=2$ and change $\sig_2$ with an element in its orbit so that $a=1$. Let $N_2$ be the $\eps$-neighborhood of $F_n \cdot (\sig_1 \cup \sig_2)$ if $N_2$ is connected we are done, otherwise continue choosing representatives $\sig_3 , \dots , \sig_k$ such that the convex hull $W = Hull(\cup_{i=1}^k \sig_i)$ has total volume $\leq \sum_{i=1}^k \theta_i + k \eps$. 

Each component of $V = T \setminus F_n \cdot \bigcup_{i=1}^k \sig_i$ is a tree. There are finitely many orbits of connected components of $V$. Each orbit of connected components of $V$ meets $W$. Let $J_1, \dots , J_m$ be connected components of $V$, that meet $W$ and belong to distinct orbits and so that the orbit $W \cup \bigcup_{i=1}^m J_i$ is equal to $T$. 
For each $i$, $J_i$ is a tree and let $H_i < F_n$ be the stabilizer of $J_i$ in $F_n$. Then $H_i$ is a finite rank free group. Since $B \cap J_i$ is dense in $J_i$, then the action of $H_i$ on $J_i$ has dense orbits. By \cite{LL} there exists a basis $\mathcal{B}_i$ of $H_i$ and a point $y_i \in J_i$ such that $d(y_i,W)<\eps$ and $d(y_i, h \cdot y_i)<\eps$ for all $h \in \mathcal{B}_i$. Define
\[
U = Hull \left( W \cup \bigcup_{i =1}^m \{ h \cdot y_i \mid h \in \mathcal{B}_i \} \right) \]

We note that $U$ is a finite tree and the volume of $U$ is approximately equal to that of $W$ which is approximately equal to the volume of the sum of lengths of $\sig_1, \dots, \sig_k$. 
\begin{eqnarray*} 
vol(U) & \leq & vol(W) + \sum_{i=1}^m (d(y_i,W) + \sum_{y \in \mathcal{B}_i}d(y_i, h \cdot y_i)) \\
& \leq & vol(W) + \sum_{i=1}^k (|\mathcal{B}_i|+1) \eps \\
& \leq & \sum_{i=1}^k \theta_i + \left( k + \sum_{i=1}^k(|\mathcal{B}_i|+1) \right) \eps 
\end{eqnarray*}

We claim that the translates of $U$ cover $T$. Indeed, the translates of $Hull\{ h \cdot y_i \mid h \in \mathcal{B}_i \cup \{1\} \}$ cover $J_i$ and the orbit of the union of $J_i$ cover $V$ which is the complement of translates of the segments $\sig_1,\dots, \sig_k$. Therefore, $qvol(T) \leq vol(U) \leq \sum_{i=1}^k \theta_i + \left( k + \sum_{i=1}^k(|\mathcal{B}_i|+1) \right) \eps$. But since the $\sig_i$'s belong to distinct orbits, $qvol(T) \geq \sum_{i=1}^k \theta_i$. In conclusion, $qvol(T) = \sum_{i=1}^k \theta_i$ and is within a multiple of $\eps$ away from $vol(U)$. 

We define the tree $T'$ to be the result of collapsing the sets $g \cdot J_i$ for each $i \in \{1,\dots, m\}$ and $g \in F_n$. Let $p \from T \to T'$ be the quotient map. Then $T'$ is a simplicial tree whose edges correspond to translates of the $\sig_i$'s. Let $U' = p(U)$ then $U'$ is connected. The set 
$P' = \{ g \in F_n \mid |g|>0 \text{ and } gU' \cap U' \neq \emptyset \}$ is finite. Let 
\[ \mathcal{S} = P' \cup \bigcup_{i=1}^k \mathcal{B}_i. \]
It is elementary to show that $\cal{S}$ generates the set $P = \{ g \in F_n \mid gU' \cap U' \neq \emptyset \}$. Since $P$ generates $F_n$ (Proposition \ref{howtogen}) , then $\cal{S}$ generates $F_n$. Moreover, it is also clear that $g \cdot U \cap U \neq \emptyset$ for each $g \in \cal{S}$. We leave the details to the reader. 
\end{proof}

\begin{prop}\label{forwardDirection}
If $X$ is a limit of a forwards Cauchy sequence in $\os$ then it has unit quotient volume.
\end{prop}
\begin{proof}
Let $x_m$ be a forward Cauchy sequence and $X$ its limit. 
For all $\eps>0$ there is an $N(\eps)$ so that for $m>N(\eps)$, $\Lip(f_{m,\infty}) < 1+\eps$. Therefore, $qvol(X) \leq (1+ \eps) qvol(X_m) = 1+\eps$, since $\eps$ was arbitrary, $qvol(X) \leq 1$. \\
\indent To show the other inequality: suppose by contradiction that $vol(X)<1$. By Proposition \ref{FindUandS} there exists a finite 
subtree $U$ in $T$
such that $vol(U) = c <1$ and a finite generating set $\calS$ of $F_n$ so that for all 
$g \in \calS$, $gU \cap U \neq \emptyset$. 
Suppose that $U$ is a union of $k$ non-overlapping 
segments. Let $\eps = \frac{1-c}{6nk}$ and assume $m$ is 
large enough so that there is a set $U' \subset X_m$ with an $\calS$-invariant 
$\eps$-relation to $U$. 
Let $\theta_1, \dots, \theta_k$ be the lengths of the arcs $\sig_1, \dots, \sig_k$ of $U$. Then there exist arcs $\tau_1, \dots \tau_k$ in $U'$ so that 
$|len(\tau_i) - len(\sig_i)'|< \eps$ for each $i$. 
If $\tau_1, \dots , \tau_k$ do not cover $U'$ let $\tau_{k+1}, \dots , \tau_J$ be the remaining arcs. Since the valence at each vertex is no greater than $3n-3$ we have $J<(3n-3)k$. Moreover, since the $\eps$-relation is onto, the length of $\tau_j$ for $k+1\leq j \leq J$ is no greater than $\eps$. Therefore,
\[ \vol{U'} \leq \sum_{i=1}^J len(\tau_i) \leq \sum_{i=1}^k len(\tau_i) + J \eps \leq \vol U + (k+J)\eps<1 \]
However, for each $g \in \calS$, $gU' \cap U' \neq \emptyset$. Therefore, $qvol(X_m) \leq vol(U')<1$ which is a contradiction. 
\end{proof}

\begin{prop}\label{forwardDirection2}
If $X$ is a limit of a forwards admissible sequence in $\os$ then it admits no arc with a non-trivial stabilizer. 
\end{prop}
\begin{proof}
Let $\{X_m\} \subset \os$ be a forwards admissible sequence and $X$ its limit. 
The idea is 
that an arc stabilizer will take up a definite part of the volume which would lead to 
$X_m$ having less than unit quotient volume. 

Let $X'$ be the simplicial tree provided in Proposition \ref{FindUandS} and let $\theta$ be the length of the smallest edge in $X'$. We take $\eps$ to be smaller than $\theta/3$ and let $U$ be a finite subtree of $U$ so that $\vol U< 1 + \eps$ and let
$\calS \subset F_n$ be a generating set such that for each $g \in \calS$, $gU \cap U \neq \emptyset$. 

Suppose $U$ contains an arc 
$\nu$ and $a \in F_n$ such that $a \cdot \nu = \nu$. 
A segment with a non-trivial 
stabilizer is not contained in a dense subtree (see e.g. \cite[Lemma 4.2]{LL}). Thus we may choose $len(\nu,X) \geq \theta$. 
Let $U' \subset X_m$ be a set with an $\calS \cup \{a\}$ equivariant $\eps$-relation to 
$U$. 
As in the proof of Proposition \ref{FindUandS} we have $vol(U') < qvol(X) + (3n-2)k \eps$. 
Let $\nu'=[p,q]$ be a segment approximating $\nu$ in $U'$, then 
$len(\nu') \geq \theta - \eps$. 
We claim that $len(a \nu' \cap \nu') > len(\nu') - 2 \eps$. 
Let $p,q$ be the endpoints of $\nu'$, then the 
segments $[p,ap]$ and $[q,aq]$ 
have length bounded above by $\eps$ choosing $\eps< \frac{len(\nu')}{3}$
we have $[p,ap] \cap [q,aq] = \emptyset$. Let $[m,n]$ be the bridge between $[p,ap]$ and $[q,aq]$. 
Then $len([m,n])> len(\nu') - 2 \eps > \eps$ and $[m,n] = \nu' \cap a \nu'$. 
Since the action of $a$ is hyperbolic, both segments $[p,ap]$ and $[q,aq]$ intersect $A_{X_m}(a)$. Therefore $[m,n] \subset A_{X_m}(a)$ and $l(\rho_m(a),X_m)< \eps$. 
Thus we may chop off most of the segment $[m,n]$ in $U'$ leaving a segment of length $ \leq \eps$, and get a set $U''$ whose translates cover $X_m$. However, 
\begin{eqnarray*}
vol(U'') & \leq &vol(U') - ( len[m,n] - \eps) \\ & \leq & (vol(U) + (3n-2)k \eps) - ( (\theta - 3\eps) - \eps) 
< 1- \theta + (3n+3)k \eps. \end{eqnarray*}
Therefore, if $\eps< \frac{\theta}{2(3n+3)k}$ then $vol(U'')<1$ and $X_m$ is covered by translates of $U''$ which is a contradiction to $qvol(X_m) = 1$. Hence there are no edges with non-trivial stabilizers. 
\end{proof}

A very small $F_n$ tree $T$ gives rise to a graph of actions. 
\begin{defn}\label{graphOfDefn}
\cite{Lev} A \emph{graph of actions} $\mathcal{G}$ consists of 
\begin{enumerate}
\item a \emph{metric graph of groups} which consists of the following data: a metric graph $G$, with vertex groups $H_v$, and edge groups $H_e$ and injections $i_e: H_e \to H_v$ when $v$ is the initial point of the oriented edge $e$.
\item for every vertex $v$, an action of $H_v$ on an $\RR$-tree $T_v$.
\item for every oriented edge $e$ with $v = i(e)$ there is a point $q_v \in T_v$ which is fixed under the subgroup $i_{e}(H_e)$. 
\end{enumerate}
\end{defn}
Let $T$ be a very small $F_n$ tree, one constructs the simplicial $F_n$-tree $T'$ and $p: T \to T'$ as described in Proposition \ref{FindUandS}. The graph $G$ is the quotient of $T'$ by the $F_n$-action. Lifting an edge $e$ of $G$ to $e'$ in $T'$ define $H_e = \text{Stab}(e')$ and $H_{i(e)} = \text{Stab}_{F_n}(i(e'))$. For a vertex $v$ in $G$ let $v'$ be some lift of $v$ and let $T_v$ be the preimage of $v$ in the tree $T$. The point $q_e$ of $T_v$ is the point $T_v \cap p(e')$. 
Conversely, given a graph of actions, one can combine the universal cover of $G$ with the trees $T_v$ to obtain a very small $F_n$-tree \cite[Theorem 5]{Lev}.

\begin{prop}\label{backDir}
If $T$ is a very small $F_n$ tree with unit quotient volume and no arc stabilizers then there exists a forwards Cauchy sequence $\{X_m\} \subset \os$ such that $\lim_{m \to \infty} X_m = T$. 
\end{prop}
\begin{proof}
Let $\GG$ be the Levitt graph of actions of $T$. Since all edge groups are trivial, 
all vertex groups are free factors. Let $V$ be the set of vertices of $\GG$ with non-trivial vertex 
groups. For each $v \in V$ there is a tree $R_v$ in $T$, invariant under (a conjugate of) the vertex group $H_v$ 
and so that $H_v$ acts on $R_v$ with dense orbits. Levitt and Lustig \cite{LL} show that 
for every $\eps$ there is a free simplicial tree $S_{v,\eps}$ that admits a 1-Lipschitz equivariant map $f_{v,\eps} \from S_{v,\eps} \to R_v$ and $qvol(S_{v, \eps}) = \eps$. For an edge $e$ in $G$ (the underlying graph of $\GG$) let $q_e^\eps = f_{i(e),\eps}(q_{i(e)})$. 
We construct a new graph of actions $\GG_\eps$ by replacing $R_v$ and $q_e$ in $\GG$ with $S_{v,\eps}$ and $q_e^\eps$. 
The resulting graph of actions $\GG_\eps$ can be developed to a simplicial tree $X_\eps$ with quotient volume equal to $1+M \eps$ where $M$ is bounded uniformly by $n$. 
There is an equivariant $1$-Lipschitz map $f_{\eps}$ from $X_\eps$ to $X$ that restricts to $f_{v,\eps}$ on each tree $S_{v,\eps}$. Moreover, for 
$\eps'<\eps$ there is an equivariant $1$-Lipschitz map $f_{\eps,\eps'}$ from $X_{\eps}$ onto $X_{\eps'}$. 
To make $X_\eps$ have unit quotient volume we replace $X_\eps$ with $X_\eps'$ by rescaling the edges (outside of the dense trees) by $1/\left( 1+ \sum_{i=1}^{|V|} qvol(S_{v,\eps}) \right)$. Thus $qvol(X_\eps') = 1$. Moreover, define $f_{\eps,\eps'}'$ to be the affine map induced by $f_{\eps,\eps'}$ and $f_\eps'$ the map induced by $f_\eps$. Then $\lip{f_\eps'} \leq 1+|V|\eps$ and also $\lip{f_{\eps,\eps'}' }\leq 1+|V|\eps$. Therefore, $\{X_{1/m}\}$ is a forwards Cauchy sequence and $X$ is its forward limit.
\end{proof}

We can now prove Theorem \ref{myB} from the introduction. 
\begin{proof}[Proof of Theorem \ref{myB}]
This follows from Propositions \ref{forwardDirection}, \ref{forwardDirection2} and Proposition \ref{backDir}.
\end{proof}

\section{Description of the distance in the completion}

The distance in the completion is defined in Definition \ref{defHat}. We will show that it can also be defined as follows: Let $\Cos$ to be the set of very small $F_n$-trees with unit volume and no non-trivial arc stabilizers. For $X,Y \in \Cos$ we let 
\begin{equation}\label{Cosdistance} 
d'(X,Y) = \log \sup \left. \left\{ \frac{l(g,Y)}{l(g,X)} \right| g \in F_n \right\}
\end{equation} 
Notice that for $X,Y \in \os$, $d'(X,Y) = d(X,Y)$.

\begin{prop}\label{propd'}
The function $d' \from \Cos \times \Cos \to \RR\cup \{\infty\}$ satisfies:
\begin{enumerate}
\item The directed triangle inequality.
\item If $d'(X,Y) = d'(Y,X) = 0$ then $X=Y$. 
\item If there exists $g \in F_n$ with $l(g,X)=0$ and $l(g,Y)>0$ then $d'(X,Y) = \infty$. 
\item If there exists an $L$-Lipschitz equivariant map $f \from X \to Y$ then $d'(X,Y) \leq \log L$. 
\end{enumerate}
\end{prop}
\begin{proof}
\begin{enumerate}
\item This is obvious from the properties of $\sup$. 
\item If $X \neq Y$ then their length functions are different. Therefore, there exists $g \in F_n$ such that $l(g,X)<l(g,Y)$ or $l(g,Y)< l(g,X)$ which implies, $d(X,Y) \neq 0$ or $d(Y,X) \neq 0$.
\item This is obvious.
\item For each $p,q \in X$ we have $d(f(p), f(q)) \leq L d(p,q)$. Therefore, for each $g$ that is eliptic in $X$ - it is also eliptic in $Y$. Moreover, if $g$ is hyperbolic in $X$ with translation length $l(g,X)$ then $l(g,Y) \leq L l(g,X)$. $\qedhere$
\end{enumerate}
\end{proof}

\begin{prop}\label{distIsLip}
For every $X,Y \in \Cos$, $d'(X,Y) = \hat d(X,Y)$. 
\end{prop} 
\begin{proof}
By forwards continuity, $\hat d = d = d'$ for $X,Y \in \os$. 
Let $\{ X_m \}_{m=1}^\infty, \{Y_k \}_{k=1}^\infty$ be Cauchy sequences in $\os$ such that $X=\lim_{m\to \infty}X_m$ and $Y=\lim_{k \to \infty}Y_k$ as length functions. 
We need to prove: 
\begin{enumerate}
\item $d'(X,Y) = c< \infty$ if and only if for all $\eps>0$ there exists an $N=N(\eps)$ such that for all $m>N$ there is a $K=K(m,\eps)$ such that $| d(X_m,Y_k) -c |<\eps$ for all $k>K$.
\item\label{case2} $d'(X,Y) = \infty$ iff for all $r$ there is an $N(r)$ such that for al $m > N(r)$ there is a $K(m,r)$ such that $d(X_m, Y_k)> r$ for all $k>K$. 
\end{enumerate}
By the triangle inequality we have $d'(X,Y) \geq d'(X_m,Y) - d'(X_m,X)$, thus by Proposition \ref{propd'}(4) for large enough $m$, 
\begin{equation}\label{triangle}
d'(X,Y) \geq d'(X_m,Y) - \eps 
\end{equation}
Since there exists an equivariant map $X_m \to Y$, the distance $d'(X_m, Y)$ is finite (by Proposition \ref{propd'}(4) ).
Fix $m$, and let $\beta_1, \dots ,\beta_s$ be the list of candidates of $X_m$. 
Choose $K= K(m,\eps)$ large enough so that for all $k>K$: 
\begin{enumerate}
\item if $l(\beta_i, Y) = 0$ then $l(\beta_i, Y_k) < injrad(X_m)$ and 
\item if $l(\beta_i,Y) >0$ then $|l(\beta_i,Y_k) - l(\beta_i,Y)| < \eps l(\beta_i,Y)$. This is possible since the list of candidates in $X_m$ is finite, and $\lim_{k \to \infty} l(\beta_i, Y_k) = l(\beta_i,Y)$.
\end{enumerate} 
Let $\gamma$ be a candidate of $X$ that is elliptic in $Y$, by item (1) 
$\frac{l(\gamma,Y_k)}{l(\gamma,X_m)}<1$, so $\gamma$ cannot 
realize the distance $d(X_m,Y_k)$. 
Let $\delta$ be the candidate that is a witness to the distance $d(X_m,Y_k)$ 
then by item (2), \[l(\delta,Y_k) \leq (1+\eps) l(\delta,Y)\] 
dividing by $l(\delta,X_m)$ and taking the log we get $d(X_m,Y_k) \leq \log(1+\eps) + d(X_m,Y)$. Combining this with inequality (\ref{triangle}), there is a constant $C$ so that for all large enough $m,k$: 
\begin{equation}\label{eqFirstDir}
d'(X,Y) \geq d'(X_m,Y_k) - C \eps.
\end{equation}

If $d'(X,Y) < \infty$ then for all $\eps$ there is some $b \in F_n$ that $\eps$ approximates the distance in equation \ref{Cosdistance}.
Thus, $l(b,X)>0$ and let $N(\eps)$ be such that for all $m>N(\eps)$, 
$|l(b, X_m) - l(b,X)|< \eps l(b,X)$. 
Thus \[\frac{l(b,Y)}{l(b,X_m)} \geq \frac{l(b,Y)}{(1+\eps) l(b,X)}\] 
which implies $d'(X,Y) \leq d'(X_m,Y) + \log(1+\eps)$ for $m>N(\eps)$. By the 
triangle inequality, $d'(X_m,Y) \leq d'(X_m,Y_k) + d'(Y_k,Y)$. Thus 
for $K(\eps)$ with the property that $k>K(\eps)$ implies $d'(Y_k,Y)<\eps$, we have 
\begin{equation}\label{eqOtherDir}
d'(X,Y) \leq d'(X_m,Y_k) + \eps + \log(1+\eps)
\end{equation}
for all $m>N(\eps)$ and $k>K(\eps)$. Thus, if $d'(X,Y)<\infty$ then equations \ref{eqFirstDir} and \ref{eqOtherDir} imply $d'(X,Y) = \hat d(X,Y)$. 

If $d'(X,Y) = \infty$ then either there is some $\beta$ so that $l(\beta,X) = 0$ and $l(\beta,Y) >0$, or for all $r>1$ there is some $\beta$ in $X$ so that $\frac{l(\beta,Y)}{l(\beta,X)} > 2r$. If the former occurs, then there exist $N$ and $K$ so that for $m>N,k>K$ we have $l(\beta,X_m)<\frac{l(\beta,Y)}{r }$ and $l(\beta,Y_k) \geq (1 - \frac{1}{r})l(\beta,Y)$ thus $\frac{l(\beta,Y_k)}{l(\beta,X_m)}\geq \frac{(1-\frac{1}{r})l(\beta,Y)}{\frac{1}{r} l(\beta,Y)} \geq r-1$ and $d(X_m,Y_k) > \log(r-1)$. If the latter occurs, then $l(\beta,X), l(\beta,Y) >0$ and there are $N,K$ large enough so that for all $m>N, k>K$ we have $l(\beta,X_m) \leq (1 + \frac{1}{r})l(\beta,X)$ and $l(\beta,Y_k) \geq (1 - \frac{1}{r})l(\beta,Y)$. Thus $\frac{l(\beta,Y_k)}{l(\beta,X_m)} \geq \frac{(1-\frac{1}{r}) l(\beta,Y)}{(1+\frac{1}{r})l(\beta,X)} \geq r$ and $d(X_m,Y_k)>\log r$ for $m>M$ and $k>K$.
In both cases we have shown item (\ref{case2}).
\end{proof}

\begin{notation}
From now on we denote both $\hat d$ and $d'$ by $d$. 
\end{notation}

When $T$ is any tree, the supremum in the formula for the distance might not be realized. We now show that if $T$ is simplicial then the supremum is realized and can be obtained by taking a maximum on a finite set of conjugacy classes which we call candidates. We first recall the definition of the Bass Group of a graph of groups.

\begin{defn}[The Bass group of a graph of groups.]
Given a graph of groups $\GG = (G, \{ G_v\}_{v \in V}, \{H_e\}_{e \in E}, \{ i_e: H_e \to G_{ter(e)} \})$ with $V$ be the set of vertices of $G$ and $E$ the set of edges of $G$ (see Definition \ref{graphOfDefn}). Let $F_E = F(\{t_e \mid e \in E\})$ denote the free group on the basis $E$, the Bass group of $\GG$, $\mathcal{B}(\GG)$ is the quotient of the free product
\[*_{v \in V} G_v * F_E / R \] where $R$ is the normal subgroup generated by 
\begin{enumerate}[a)]
\item $t_e^{-1} = t_{\bar e}$
\item $t_e i_e(g) t_e^{-1} = i_{\bar e}(g)$ for all $e \in E$ and $g \in H_e$.
\end{enumerate} 
A connected word $w$ in the Bass group has the form 
$w=r_0 t_1 r_2 t_2 \dots t_q r_q$ for which there is an edge path $e_1 \dots e_q$ in $G$ such 
that $r_0 \in G_{ini(e_i)}$, $r_i \in G_{ter(e_i)}$ and $t_i = t_{e_i}$. The word 
$w$ is a cyclic word if the edge path it follows is a loop. 
$w$ is reduced if either $r_0 \neq 1$ and $q=0$ or, $q>0$ and $w$ does not contain a subword of the form $t_e i_e(g) t_e^{-1}$. 
The \emph{fundamental group of the graph of groups} $\pi(\GG,P)$ based at the point $P \in G$ is the subgroup of $\mathcal{B}(\GG)$ of all cyclic subwords based at $P$. See \cite{CohL} for more detailed definitions.
\end{defn}

\begin{defn}\label{can}
A candidate $\al$ in a marked metric graph of groups $x$ is a cyclically reduced cycle word of the Bass group that follows a path of the following type:
\begin{enumerate}
\item an embedded loop
\item an embedded figure 8
\item a barbell
\item a barbell whose bells are single points 
\item a barbell which has one proper bell and one collapsed bell. 
\end{enumerate}
\end{defn}

The following Proposition is a generalization of what happens in Outer Space. 

\begin{prop} \label{ConstructLipMap}
If $S$ is simplicial and $T \in \Cos$ let 
\[ M_{stretch} = \max \{ st(\al,S,T) \mid \al \text{ a candidate}\}\] 
and let 
$m_{lip}= \infty$ if no equivariant Lipschitz map $S \to T$ exists, and otherwise 
\[m_{lip} = \min \{ \Lip(h) \mid h:S \to T \text{an equivariant Lipschitz map} \}.\] Then
\[ d(S,T) = \log M_{stretch} = \log m_{lip} \]
\end{prop}
\begin{proof}
We wish to show that if one of the quantities in the equations is infinite then so is the other. 
The proof that the quantities are equal in the case that they are finite is identical to the case both trees are in $\os$ and can be found in \cite{FM}.

If there exists some equivariant Lipschitz map $f: S \to T$ then $M_{stretch} \leq \Lip(f)$. Thus, if $M_{stretch} = \infty$ then there is no Lipchitz map $S \to T$. 
Conversely, 
suppose that $M_{stretch} < \infty$ so in particular, all of 
the elliptic elements of $S$ are also elliptic in $T$. 
More generally, if $H<F_n$ is elliptic in $S$ then it must be eliptic in $T$. Indeed suppose $g,g'$ both fix $p \in S$ and $g$ fixes $q$ in $T$ and $g'$ fixes $q'$ in $T$. Then $gg'$ is eliptic in $S$ (fixing $p$) but in $T$, since there are no fixed arcs, it is hyperbolic. 
We wish to construct an equivariant map from $S$ to $T$. 
Let $g \in F_n$ be eliptic in $S$ fixing $p \in S$. Suppose $g$ fixes $q$ in $T$. Define $f \from S \to T$ by setting $f( a \cdot p) = a \cdot q$ for all $a \in F_n$ and extending $f$ piecewise linearly on edges (there are some choices here). We claim that $f$ is well defined. Indeed, if $a \cdot p = b \cdot p$ then $a^{-1}b$ stabilizes $p$. Therefore, $a^{-1}b \in Stab_S(p) \subset Stab_T(q)$. Hence $a \cdot q = b \cdot q$. In conclusion, $M_{stretch} < \infty$ implies the existence of a Lipschitz map and hence $m_{lip}< \infty$. 
\end{proof}

\begin{question}
Does Proposition \ref{ConstructLipMap} hold even when $S$ is not simplicial?
\end{question}

When $T$ is a non-simplicial, there is a collapse map to a simplicial tree 
$T \to T'$ so that $d(T,T') = 0$ see Proposition \ref{FindUandS}. The next 
proposition shows that this cannot happen when $T$ is simplicial. This will be important in the proof that the simplicial completion is invariant under an isometry. 

\begin{prop}\label{zero_distance}
If $X$ simplicial and $Y \in \Cos$ so that $d(X,Y) = 0$ then $X=Y$. 
\end{prop}
\begin{proof}
By Proposition \ref{ConstructLipMap}, there is an equivariant map 
$f:X \to Y$ that is 1-Lipschitz. The map $f$ is onto (since $Y$ is minimal). Any map 
can be homotoped without increasing its Lipschitz constant, so that the restriction of 
the new map on each edge is an immersion or a collapse to one point. No edge is 
stretched since $\Lip(f) = 1$. No edge is shrunk or collapsed because 
$qvol(X) =1= qvol(Y)$. So $f$ restricted to each edge is an isometry. 
We claim that $f$ is an immersion which would finish the proof. 
Indeed if $f$ is not 
injective then $f(p) = f(p')$ and $f|_{[p,p']}$ is not immersed. Suppose $f$ is not 
an immersion at a neighborhood of $v$. Then there are two edges $e_1, e_2$ incident at $v$ so that 
$f(e_1), f(e_2)$ define the same germ. If $e_1, e_2$ are not in the same orbit then 
this would contradict $qvol(X) = qvol(Y)$ by Proposition \ref{forwardDirection} ($f$ loses a definite part of the volume). 
So assume there is a $g \in F_n$ such that $g \cdot e_1 = e_2$ (with the appropriate 
orientation). But then $f( g \cdot e_1 ) = f(e_1)$ so $g$ stabilizes a non-trivial segment in 
$Y$ which contradicts Proposition \ref{forwardDirection2}. Hence $f$ is a surjective isometric immersion 
i.e. an isometry. In particular, $Y$ is simplicial. \end{proof}

\section{Topologies on the simplicial metric completion}

\begin{defn}\label{defnSplittingComplex}
A \emph{free splitting} of $F_n$ is a simplicial tree $S$ along with a simplicial $F_n$ action 
on $S$, so that this action is minimal, non-trivial, irreducible, and edge stabilizers are 
trivial. The quotient graph $S/F_n$ is a finite graph and we say the \emph{covolume} 
of $S$ is the number of edges in this quotient graph. Two free splittings $S,S'$ are 
compatible if there is a free splitting $S''$ and equivariant edge collapses $S'' \to S$ 
and $S'' \to S'$. \\
The \emph{free splitting complex} $\fsc$ is the simplicial complex whose set of 
vertices is the set of free splittings of covolume 1 and there is a simplex spanned by 
$S_1, \dots S_k$ if they are pairwise compatible. This turns out to be equivalent to 
the existence of a splitting $S$ and maps $S \to S_i$ for $i=1, \dots , k$ that are 
equivariant edge collapses.
\end{defn}

Outer Space $\os$ naturally embeds in $\fsc$ as a subcomplex with missing faces. If we 
add the missing faces we obtain $\fsc$. Therefore, $\fsc$ is called the \emph{simplicial 
completion} of Outer Space. 

\begin{defn}
The (metric) simplicial completion $\scos$ is the set of points $X \in \Cos$ such that $X$ is simplicial. 
\end{defn}

By the characterization of the completion points in the boundary, Theorem \ref{myB}, the set of simplicial trees in the completion is the same as the simplicial completion, i.e. $\fsc$. In this section we compare two natural topologies on this set. 

\begin{defn}[The Euclidean topology on $\fsc$]
Let $\sig$ be a simplex in $\fsc$ and let $d_1(x,y)$ be the $l_1$ metric for $x,y \in \sig$ (using the edge lengths).
We define $B_\sig(x,\eps) = \{ y \mid d_1(x_i,y) <\eps \}$. The Euclidean ball around $x$ is $B_{Euc}(x,\eps) = \cup_{x \in \sig} B_\sig(x,\eps)$. The Euclidean topology is the topology generated by Euclidean balls. 
\end{defn}

There is also a natural Lipschitz topology on $\scos$.

\begin{defn}[The (incoming-)Lipschitz topology on $\scos$]
A basis for the Lipschitz topology is the collection of incoming balls, see Definition \ref{defBInOut}.
\end{defn}

The Lipschitz topology will be preserved under isometries of Outer Space. We show that the Euclidean topology coincides with the Lipschitz topology.

\begin{remark}The topology generated by the ``outgoing" balls is different from the Euclidean topology. Consider a point $x$ in the completion so that the underlying graph of the quotient $x/F_n$ is a single, non-separating edge (a one edge loop). For such $x$ and for all $x \neq y \in \scos$, $d(x,y) = \infty$. Hence the only open sets of the outgoing-Lipschitz topology containing $x$ are $\{ x \}, \scos$. 
\end{remark}

\begin{notation}
For $x \in \scos$ we denote $\theta(x)$ the smallest edge of $x$, and by $\inj(x)$ the 
injectivity radius of $x$. Also notice that we are considering two metrics here - the Lipchitz metric $d(x,y)$ and the Euclidean metric $d_1(x,y)$.
\end{notation}

\begin{lemma}\label{lem2}
The function $d( \cdot, x)$ is continuous with respect to the Euclidean topology on $\scos$. 
\end{lemma}
\begin{proof}
For $y \in \scos$ let 
$\delta = \min \left\{ \theta(y)\eps, \frac{\inj(y) \eps}{6n-6} \right\}$, 
we claim that for $y' \in B_{Euc}(y,\delta)$ we have $|d(y',x) - d(y,x)|< \eps$. 
Let $y' \in B_{Euc}(y,\delta)$ and let $\sig$ be a simplex containing $y,y'$, and 
enumerate the edges in the graph corresponding to $\sig$ by $e_1, \dots , e_J$. 
Thus, $y=(y_1, \dots , y_J), y' = (y_1', \dots , y_J')$ in the $\sig$ coordinates, and 
$|y'_i - y_i|< \delta$. By our choice of $\delta$, for $i$ so that $y_i>0$, $\frac{y_i}{y_i'}< \frac{y_i}{y_i-y_i\eps}= \frac{1}{1-\eps}$. Therefore,
\begin{equation}\label{eq1024} 
d(y',x) \leq d(y',y) + d(y,x) \leq \log\left( \frac{1}{1-\eps}\right) + d(y,x) 
\end{equation}
Let $\al$ be a realizing candidate for $d(y,y')$. Since $|y_i'-y_i|<\delta$ and $\al$ contains no more than $2(3n-3)$ edges we have
\[ \frac{l(\al,y')}{l(\al,y)} \leq \frac{l(\al,y) + 2(3n-3)\delta}{l(\al,y)}< 1+ \eps \]
Thus $d(y',x) \geq d(y,x) - d(y,y') \geq d(y,x) - \log(1+\eps) $. Combining with equation (\ref{eq1024}) we conclude the proof.
\end{proof}

\begin{cor}\label{EucFinerThanLip}
The Euclidean topology is finer than the Lipschitz topology.
\end{cor}

To show that they are in fact equal we need the following two lemmas. 

\begin{lemma}\label{lem1}
For every simplex $\sig$ in $FS_n$ 
\begin{enumerate}
\item For every $x \in \sig$ and for every face $\tau$ of $\sig$ so that $x \notin \tau$ there is an $\eps(x,\tau)>0$ such that $d(\tau,x)>\eps$.

\item Let $x \in \sig$, for all $r>0$ there is a $t(x,\sig,r)$ such that if $y \in \sig$ and $d(y,x)<t(x,\sig,r)$ then $d_1(y,x)<r$. 
\end{enumerate}
\end{lemma}
\begin{proof}
To prove part (1), $\tau$ is compact in the Euclidean topology, and $d(\cdot,x)$ is continuous.
Thus $d(\cdot,x)$ achieves a minimum $\eps(x,\tau)$ on $\tau$.
If $\eps(x,\tau) = 0$ then there is a $y \in \tau$ with $d(y,x)= 0$ and by Proposition \ref{zero_distance} $y=x$ but $x \notin \tau$. Therefore $\eps(x,\tau) >0$. \\
To prove part (2), consider for $r>0$,
\[A(x,r) = \{y \in \sig \mid d_1(y,x) \geq r \}\] 
Since $A(x,r)$ is compact in the Euclidean topology, then there exists a minimum $t(x,\sig,r)$ to the set $\{d(a,x) \mid a \in A(x,r) \}$. As before, $t \neq 0$ since $x \notin A(x,r)$. Therefore, $d(y,x)<t$ and $y \in \sig$ implies that $d_1(y,x)<r$. 
\end{proof}

\begin{lemma}\label{InSimplex}
For every $x \in \scos$ there is a constant $\eps(x)>0$ such that 
for all $y \in \scos$ with 
$d(y,x)< \eps$ there exists a simplex $\sig \in \fsc$ that contains both $x$ and $y$. 
\end{lemma}
\begin{proof}

Let $x$ be contained in the interior of the simplex $\tau$. By Lemma \ref{lem1} for 
any simplex $\sig \supseteq \tau$ and for any face $\tau'$ of $\sig$ that does not 
contain $x$ there is an $\eps = \eps(x,\tau')$ so that $d(\tau',x)> \eps$. 
We show that we can find such an $\eps$ independent of $\tau'$. The difficulty is that the link of $x$ is potentially infinite. 
Recall that $\out$ acts cocompactly on $\fsc$ by simplicial automorphisms.
For an automorphism $\phi \in \out$, $\eps(\phi(x), \phi(\tau)) = \eps(x,\tau)$. For $H<\out$ a finite index torsion free subgroup, the quotient $\fsc/H$ is a finite CW-complex, therefore, there are finitely many isometry types of simplicies $\sig$ containing $x$. Therefore the set
\[ \{ \eps(x,\sig,\tau) \mid x \in \sig, x \notin \tau \subset \sig \}\]
is finite and therefore achieves a minimum $\eps(x)$. 

Let $y \in \scos$ such that $d(y,x)< \infty$ and there is no simplex 
containing both then $d(y,x)>\frac{\eps(x)}{2}$. 
By Proposition \ref{ConstructLipMap} there is an equivariant Lipschitz 
map $f:y \to x$. 
Let $y'$ be a point in the same simplex as $y$ so that there is a Stallings fold 
sequence from $y'$ to $x$ (perturb the edges lengths in $y$ so that the stretch of the 
edges of the optimal map are all rational). Moreover, we can guarantee that 
$d(y',x)<d(y,x)+ \frac{\eps(x)}{2}$. Let 
$f':z \to x$ the last fold in the sequence. Then $z$ and $x$ are contained in the same 
simplex. Moreover, $d(y',x)>d(z,x)$, and $d(z,x) > \eps(x)$ hence $d(y,x)> 
\frac{\eps(x)}{2}$. 
\end{proof}

\begin{theorem}\label{topo}
The Lipschitz topology and the Euclidean topology on $\scos$ coincide. 
\end{theorem}

\begin{proof}
By Corollary \ref{EucFinerThanLip} it is enough to show that the Lipschitz 
topology is finer than the Euclidean topology. Let $B_{Euc}(x,r)$ be a 
ball in the Euclidean topology. By proposition \ref{InSimplex}, we may 
choose $\eps$ small enough so that $B_{\text{in}}(x,\eps)$ is contained in the star of $x$. 
By lemma \ref{lem1}, there is a $t(x,\sig,r)$ so that for all $y\in \sig$ if 
$d(y,x)<t(x,\sig,r)$ then $d_1(y,x)<r$. We need to find $t$ that works for 
all $\sig$ containing $x$. 
As in the proof of \ref{InSimplex} we use the simplicial action of $\out$ on $\fsc$. 
The quotient $\fsc/H$ where $H<\out$ is finite index and torsion free is compact. 
Since $\out$ acts by Lipschitz isometries as well as Euclidean 
isometries 
$t(x,\sig,r)$ is invariant under this action. Thus the set
\[ \{ t(x,\sig,r) \mid x \in \sig \} \] 
is finite, and therefore achieves a minimum $t(x,r)$. 
Thus, for $r>0$ let $\delta = \min\{ \eps(x), t(x,r) \}$ (where $\eps(x)$ is 
the constant from Lemma \ref{InSimplex}) then if $d(y,x)<\delta$ then there 
exists a simplex $\sig$ containing $y$ and $x$ and moreover, $d_1(y,x)<r$. 
\end{proof}

This completes the proof of Theorem \ref{myC}.

\begin{remark}
The Gromov/Axes topology on $\scos$ is strictly finer than the Lipschitz/Euclidean topology. This was included in a previous preprint but we chose to omit it for the sake of brevity. However, to see that they are not the same is quick. Consider $F_2=\langle a,b \rangle$ and let $x$ be the graph of groups with one vertex whose related group is $\langle b \rangle$ and one edge labeled $a$. Now suppose $y$ is a point in the simplex with a graph of groups which is a rose with edges labeled $ab^i, b$. Then if $y$ is in the Gromov-topology-neighborhood $U(a,\eps)$ of $x$ then the edge labeled $b$ in $y$ has to have length shorter than $\frac{\eps}{i}$. Since the length depends on $i$, it depends on $\sig$. Thus, there is no open set in the Euclidean topology that is contained in $U(a,\eps)$. 
\end{remark}

\section{The isometries of Outer Space}

\begin{prop}\label{isomExtend}
Every Lipschitz isometry $F:\os \to \os$ extends to an isometry of the completion $\hat F:\Cos \to \Cos$. The simplicial metric completion $\scos = \fsc$ is an invariant subspace and $\hat F|_{\scos}$ is a homeomorphism of $\fsc$ with the Euclidean topology.
\end{prop}
\begin{proof}
By Corollary \ref{isomX} and Proposition \ref{conditions} the isometry $F$ extends to an isometry of $\Cos$. 
We claim that $\scos$ is invariant under $\hat F$. The reason is as follows: Suppose $T \in \Cos$ is not simplicial Let $B$ be the set of branch points of $T$. Note that $\ol{B} \neq T$ since $qvol(T) = 1$. By \cite{Lev} we may equivariantly collapse the components of $\ol{B}$ in $T$ to obtain a simplicial $F_n$ tree $T'$ with $qvol(T') = 1$. 
Moreover since arc stabilizers in $T$ are trivial then in $T'$ they are trivial as well. This implies that there is a $T' \in \Cos$ such that $d'(T,T') = 0$ and $T' \neq T$. 
By proposition \ref{zero_distance} $x \in \scos$ if and only if for all $x\neq y \in \Cos$, $d(x,y)>0$. Thus $\hat F$ preserves $\scos$ and its metric $d$. By Theorem \ref{topo} the Lipschitz topology is the same as the Euclidean topology. \end{proof}

We now work to show that $\hat F|_{\fsc} \from \fsc \to \fsc$ is simplicial.

\begin{defn}
A blowup $G'$ of a graph of groups $G$ is a graph of groups $G'$ along with a marking preserving map $c \from G' \to G$ so that $c$ collapses a proper subset of $E(G')$. 
\end{defn}

\begin{lemma}\label{BlowupLemma} 
Let $G$ be a graph of groups representing a free splitting of $F_n$, for $n\geq 3$. Then either:
\begin{enumerate}
\item There are three or more different 1-edge blowups of $G$ in $\fsc$. 
\end{enumerate}
or, the graph $G$ is one of the types $(A), (B)$ or $(C)$ in Figure \ref{exTableFig}, i.e. 
\begin{enumerate}\setcounter{enumi}{1}
\item Type A: $G$ has a unique vertex $v$ with non-trivial group $G_v$. $G_v$ is cyclic and the valence of $v$ is 1 and of all other vertices is $3$. In this case there is a single 1-edge blowup of $G$ in $\fsc$ an no other possible blowups. 
\item Type B: $G$ has a unique vertex $v$ with non-trivial group $G_v$ and $G_v$ is cyclic. Moreover, there is a unique embedded circle containing $v$. The valence of $v$ is two and the rest of the vertices have valence $3$. In this case there are two 1-edge blowups and two 2-edge blowups. 
\item Type C: $G$ has precisely two vertices $v_1, v_2$ with non-trivial groups $G_{i}$ that are cyclic for $i=1,2$.
The valence of $v_1, v_2$ is one, and all other vertices have valence equal to $3$. In this case there are two 1-edge blowups and one 2-edge blowup. 
\end{enumerate}
\end{lemma}

\begin{figure}[h]
\begin{center}
\includegraphics[width=\columnwidth]{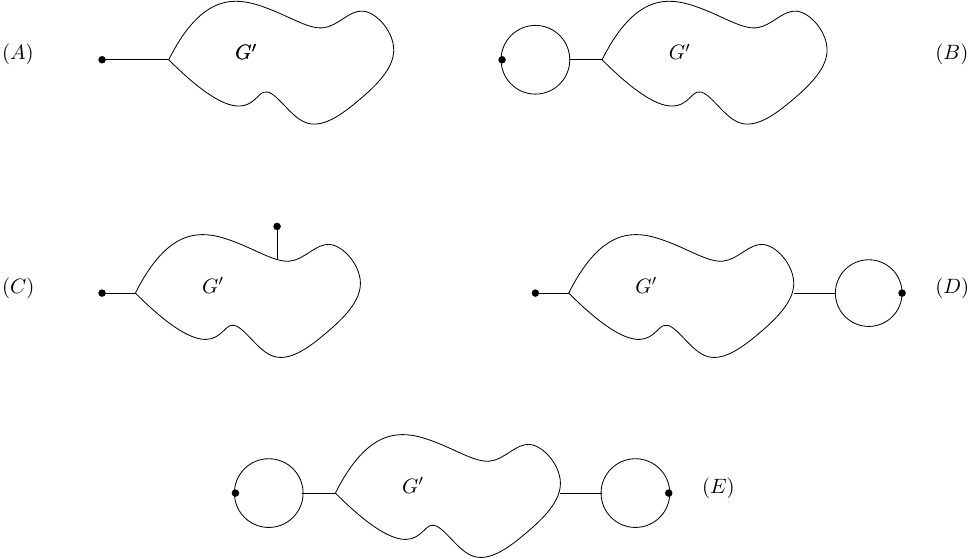}
\caption{The graphs of groups\label{exTableFig} discussed in the proof of Lemma \ref{BlowupLemma}. The subgraph $G'$ is an arbitrary trivalent graph. Types A,B and C have less than three 1-edge blowups.}
\end{center}
\end{figure}

\begin{proof}
The following types of graphs of groups have three or more 1-edge blowups.
\begin{enumerate}
\item $G$ contains a vertex of valence 4 or more.
\item $G$ contains a vertex with a non-cyclic vertex group. 
\item $G$ contains three or more vertices with non-trivial vertex groups. 
\item $G$ contains a vertex $v$ with $G_v \neq \{ 1 \}$ and there exist two distinct embedded loops $\beta, \gamma$ containing $v$. See figure \ref{example1}.
\item $G$ contains a vertex $v$ with $G_v \neq \{ 1 \}$ and a separating edge $e$ with $G - e = X \cup Y$ with $v \in X$ and $X - \{ v\} \neq \emptyset$ then again there are infinitely many 1-edge blowups of $G$ (see figure \ref{example2}).
\end{enumerate}

\begin{figure}[h]
\begin{center}
\includegraphics{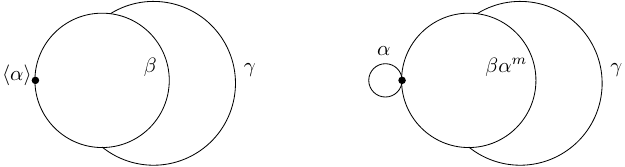}
\caption{\label{example1} The graph on the left has infinitely many 1-edge blow ups of the type on the right.}
\end{center}
\end{figure}

\begin{figure}[ht]
\begin{center}
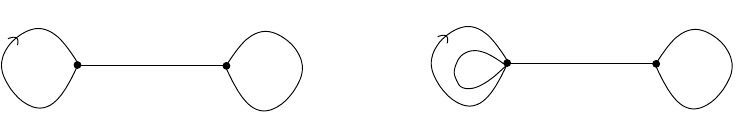
\caption{\label{example2} The graph on the left is $G_\sig$ of type (5) the graph on the right is $G_\tau$ with $\tau \supset \sig$. }
\end{center}
\end{figure}

Therefore, the remaining cases are where $G$ contains at most two vertices with non-trivial groups and the remaining vertices have valence $= 3$. Moreover, by (2) the non-trivial groups are cyclic and either $v_i$ is a valence 1 vertex or it has valence two and there is a unique embedded circle containing $v_i$. 
These are the possibilities illustrated in Figure \ref{exTableFig}. Notice that type $(D)$ has three 1-edge blowups and case $(E)$ has 4 1-edge blowups. Therefore the graphs with less than three 1-edge blowups are A,B and C.
\end{proof}

\begin{defn}
Let $X$ be a simplicial complex. A point $x \in X$ is $j$-smooth if there exists a neighborhood $U \subset X$ of $x$ that is homeomorphic to $\RR^j$. 
\end{defn}

\begin{prop}\label{jsmoothprop}
Suppose $X$ is a simplicial complex of dimension $j$. If each $j-1$ simplex is contained in either one or more than three $k+1$ simplicies then the $j$-smooth points are the interiors of $j$-simplicies.
\end{prop}
\begin{proof}
If $x \in X$ is $j$-smooth then its link is homeomorphic to $S^{j-1}$. In particular, if $\tau$ is a $j-1$ simplex in the link of $x$ then it must be contained in precisely two $j$-simplicies. 
\end{proof}

\begin{prop}\label{HomeoPresSimp}
If $\hat F$ is a simplicial self map of $\fsc$. 
\end{prop}
\begin{proof}
Francaviglia and Martino \cite{FMIsom} prove that if $F$ is a homeomorphism of $\os$ then $F$ is simplicial. They prove it by induction on 
the codimension. They consider $\os^i$ the $i$-skeleton of $\os$ and show that every 
$i-1$ simplex is attached to three or more $i$-simplicies. 
Therefore, $i$-simplicies are the connected components of the $i$-smooth set in $\os^i$ and therefore, $F$ must map each $i$-simplex to an $i$-simplex.

This statement is false for $\fsc$. Consider for example the free splitting complex for $n=2$. Consider the simplex $\sig$ corresponding to a barbell. This simplex has two codimension 1 faces $\tau_1, \tau_2$ that correspond to collapsing one of the bells, and a codimension 2 simplex $\nu$ that corresponds to collapsing both bells. It is possible to perturb the identity homeomorphism to a homeomorphism $\hat F$ of $\text{FS}_2$ that maps $\tau_1$ into the union of $\tau_1 \cup \nu_1 \cup \tau_2$ (but not into a single simplex). 

Therefore, here we must use the hypothesis that $\hat F$ preserves the Lipschitz metric on $\fsc$. 
Since top dimensional simplicies of $\fsc$ are in fact contained in $\os$ then they are preserved by \cite{FMIsom}. 
$\hat F$ restricts to a homeomorphism on the codimension 1 skeleton and preserves the part of it that is contained in $\os$. Let $\tau$ be a top dimensional simplex, then $\hat F(\bar \tau) = \ol{ \hat F(\tau)}$. Let $\sig$ be an open $(3n-5)$-simplex in $\partial \tau$. 
Then $\hat F(\sig)$ is an open set in $\partial F(\tau)$. If $\hat F(\sig)$ is not contained in a face of $\hat F(\tau)$ then there must exist two points $x,x' \in \sig$ so that $\hat F(x), \hat F(x')$ lie in different codimension 1 faces of $\hat F(\tau)$.
Two different top dimensional faces of $\hat F(\tau)$ that are not contained in $\os$ correspond to shrinking two different 1 edge loops in the graph corresponding to $\hat F(\tau)$. 
Thus, $d(\hat F(x), \hat F(x'))= \infty$
However, $d(x,x')<\infty$ since they belong to the same open simplex. We get a contradiction so $\hat F(\sig)$ must be contained in a codimension 1 face of $\hat F(\tau)$. Since $\hat F$ is invertible, $\hat F(\sig)$ is equal to a codimension 1 face. This shows that $F$ sends codimension 1 simplicies of $\fsc$ to codimension 1 simplicies and preserves the codimension 2 skeleton. 

For $j<3n-5$, by Lemma \ref{BlowupLemma} and Proposition \ref{jsmoothprop}, the connected components of the set of $j$-smooth points of the $j$ skeleton of $\fsc$ are precisely the open $j$-simplicies so $F$ preserves the set of open $j$-simplicies and the induction pushes through.
\end{proof}

\begin{cor}\label{IsomAction}
There is a homomorphism $\phi: \text{Isom}(\os) \to \text{Aut}(\fsc)$. 
\end{cor}
\begin{proof}
$\hat G \circ \hat F$ is an isometry of $\scos$ that restricts to $G \circ F$ on $\os$. By the uniqueness in proposition \ref{extIsom}, $\hat G \circ \hat F = \widehat{G \circ F}$. 
\end{proof}

\begin{cor}\label{homo}
For $n\geq 3$ there is a homomorphism $\phi: \text{Isom}(\os) \to \out$. For $n=2$, there is a homomorphism $\phi: \text{Isom}(\mathcal{X}_2) \to \text{PSL}(2,\ZZ)$.
\end{cor}
\begin{proof}
Using Corollary \ref{IsomAction}, we must show that $Aut(\fsc) \cong \out$ for $n\geq 3$ and $Aut(FS_2) \cong \text{PSL}(2,\ZZ)$. For $n \geq 3$ this is a result of 
Aramayona and Souto \cite{AS}.
For $n=2$, simplicies with free faces in 
$\text{FS}_2$ are precisely the graphs with separating edges. Thus, an automorphism 
of $\text{FS}_2$ preserves the ``non-separating splitting complex''. This complex is the Farey complex and its automorphism group is $\text{PSL}(2,\ZZ)$. We therefore get a homomorphism $\phi \from \text{Isom}\os \to \text{PSL}(2,\ZZ)$
\end{proof}

We now wish to show that the kernel of this homomorphism is trivial. 

\begin{defn}
Let $(G,\tau)$ be a marked graph representing the simplex $\sig$. For any proper subgraph $\emptyset \neq H \subset G$ let $\sig_H$ denote the face of $\sig$ obtained by collapsing each connected component of the complement of $H$.
\end{defn}

\begin{prop}\label{distDetermine}
Let $\sig$ be a simplex in $\fsc$ corresponding to the marked graph $G$. Let $\al$ be a candidate loop in $G$, let $\al_G$ denote its image. For every $x \in \text{int}(\sig)$, let $vol(\al,x)$ denote the volume of $\al$ in $x$, then 
\[ d(x,\sig_\al) = \log \frac{1}{vol(\al,x)}.\] 
\end{prop}
\begin{proof}
Denote $\lam =1/vol(\al,x)$ and let $y \in \bar\sig$ be the point such that 
\[ \left\{ \begin{array}{lr}
len(e,y) = 0 & \text{for } e \subset G-\al_G \\ 
len(e,y) = \lam len(e,x) & \text{for } e \subset \al_G
\end{array} \right. \] 
Note that the volume of $y$ is 1 since only the edges of $\al_G$ survive. 
The natural map $f:x \to y$ 
stretching the edges in $\al_G$ by $\lam$ and collapsing the others to points satisfies 
$\Lip(f) = \lam$. Therefore $d(x,y) \leq \log\lam$ and $d(x,\sig_\al) \leq \log\lam$. 

When $\al$ is a candidate that is injective except on possibly finitely many points (for example a circle or a figure 8) then in $\bar\sig$, $l(\al, \cdot) = vol(\al, \cdot)$. Thus, 
$st(\al) = \frac{1}{l(\al,x)} = \lam$ and so $d(x,z)\geq \log\lam$. Therefore, $d(x,\sig_\al) = \log\lam$.

Otherwise $\al$ has the form $\al=\beta \delta \gamma \bar \delta$ and for each $z \in \sig_\al$ we have 
\[ l(\beta,z) + l(\gamma,z) + l(\delta,z) = 1 = \lam (l(\beta,x) + l(\gamma,x) + l(\delta,x)).\] 
If either $l(\beta,z)> \lam l(\beta,x)$ or $l(\gamma,z)> \lam l(\gamma,x)$ then $d(x,z)> \log \lam$. 
Otherwise, $l(\delta,z) \geq \lam l(\delta,x)$ hence 
$l(\al,z) = 1+ len(\delta,z) \geq 1+\lam l(\delta,x) = 
\lam vol(\al,x) + \lam l(\delta,x) =
\lam l(\al, x)$. 
Hence $d(x,z) \geq \log \lam$. We thus get that $d(x,\sig_\al) = \log\lam$. 
\end{proof}

\begin{cor}\label{corCandidates}
Let $x$ be a point in the interior of the simplex $\sig$. The lengths of candidate loops of $x$ are determined by the Lipschitz distance of $x$ to the faces of $\sig$. 
\end{cor}
\begin{proof}
If $\al$ is a candidate that is injective into $x$ for all but finitely many points then $l(\al,x) = vol(\al,x)$ and the statement follows from Proposition \ref{distDetermine}. Otherwise $\al = \beta \delta \gamma \bar \delta$. Since the lengths of $\beta$ and $\gamma$ are determined by the distances to the appropriate faces and the length of $\delta$ can then be computed by the volume of $\al$. Then we can compute the length of $\al$. 
\end{proof}

We now prove Theorem \ref{myD}

\begin{theorem}[\cite{FMIsom}]
The group of isometries of $\os$ with the Lipschitz metric is $\out$ for $n \geq 3$. 
The isometry group of $\mathcal{X}_2$ with the Lipschitz metric is $\textup{PSL}(2,\ZZ) \cong \textup{Out}(F_2)/\{x_i \to x_i^{-1}\}$. 
\end{theorem}

\begin{proof}
We wish to show that the homomorphisms in Corollary \ref{homo} are injective. 
It is enough to show that if $F$ is an isometry of 
$\os$ such that that $\phi(F) = id$ 
then $F$ is the identity on $\os$. Since $\phi(F)$ is the identity then for each simplex $\sig \in \fsc$ we have $F(\sig) = \sig$.
Let $x \in \os$ and $\sig$ a simplex so that $x \in int(\sig)$. 
For all faces $\tau$ of $\sig$, $d(x,\tau) = d(F(x),F(\tau)) = d(F(x),\tau)$. By 
Corollary \ref{corCandidates}, the lengths of all candidate loops of $G$ the underlying marked graph of $\sig$, are the 
same in both $x$ and $F(x)$. Since the distance $d(x,F(x))$ is the logarithm of the maximal stretch of 
candidate loops of $x$ then $d(x,F(x))=0$ therefore $F(x) = x$ by Proposition 
\ref{zero_distance}.
\end{proof}

\bibliography{ref}

\end{document}